\documentclass{article}
\usepackage{cite}
\usepackage{amsmath,amssymb,amsfonts,amsthm}
\usepackage{algorithmic}
\usepackage{graphicx}
\usepackage{algorithm,algorithmic}
\usepackage{hyperref,pdfsync}
\usepackage{textcomp}
\usepackage{dsfont}

\usepackage{epstopdf}
\usepackage{subfigure}
\graphicspath{{./}{./figures/}}

\allowdisplaybreaks

\newtheorem{nnassumption}{\bf Assumption}

\newtheorem{nntheorem}{\bf Theorem}
\newenvironment{theorem}{\begin{nntheorem}\it}{\end{nntheorem}}
\newtheorem{nncorollary}{\bf Corollary}

\newtheorem{nndefinition}{\bf Definition}

\newtheorem{nnproposition}{\bf Proposition}
\newenvironment{proposition}{\begin{nnproposition}\it}{\end{nnproposition}}
\newtheorem{nnproblem}{\bf Problem}

\newtheorem{nnlemma}{\bf Lemma}
\newenvironment{lemma}{\begin{nnlemma}\it}{\end{nnlemma}}
\newtheorem{nnremark}{\bf Remark}
\newenvironment{remark}{\begin{nnremark} \rm }{\hfill \hspace*{1pt}\hfill $\circ$\end{nnremark}}
\newtheorem{nnexample}{\bf Example}

\renewcommand{\leq}{\leqslant}
\renewcommand{\geq}{\geqslant}

\begin{document}
\title{Controllability of wave--heat and heat--wave cascades}
\author{Hugo Lhachemi, Christophe Prieur, and Emmanuel Tr{\'e}lat
\thanks{Hugo Lhachemi is with Universit{\'e} Paris-Saclay, CNRS, CentraleSup{\'e}lec, Laboratoire des signaux et syst\`emes, 91190, Gif-sur-Yvette, France (email: hugo.lhachemi@centralesupelec.fr).}
\thanks{Christophe Prieur is with Universit\'e Grenoble Alpes, CNRS, Gipsa-lab, 38000 Grenoble, France (e-mail: christophe.prieur@gipsa-lab.fr).}
\thanks{Emmanuel Trélat is with Sorbonne Universit\'e, Universit\'e Paris Cit\'e, CNRS, Inria, Laboratoire Jacques-Louis Lions, LJLL, F-75005 Paris, France (e-mail: emmanuel.trelat@sorbonne-universite.fr).}}

\date{}

\maketitle

\begin{abstract}
We study boundary controllability of one-dimensional coupled hyperbolic--parabolic cascades, focusing on the fine structure of reachable sets.
The main model is a wave--heat cascade in which a boundary control acts on the wave equation and drives the heat equation through an internal coupling.
We provide a sharp minimal time for the hyperbolic part ($T>2L$) and a complete spectral characterization of exact controllability in weighted Hilbert spaces, whose definition depends explicitly on the coupling profile through a sequence of modal coefficients.
In particular, internal couplings may generate nonstandard highly irregular controllability spaces and yield a generic (full measure) but non-robust controllability property.
The analysis relies on Riesz basis decompositions and on an Ingham--M\"untz inequality.
We also prove that the exact controllability space is \emph{not} invariant along Hilbert Uniqueness Method trajectories: even if both endpoints belong to the controllability space, the associated minimal-energy trajectory may leave it at intermediate times.
Finally, we compare with the reversed (heat--wave) cascade and discuss how reversing the direction of the coupling transfers the loss of regularity between the parabolic and hyperbolic components.
\end{abstract}

\medskip
\noindent{\bf Keywords:}
Wave--heat cascade, heat--wave cascade, boundary control, exact controllability, HUM, Riesz basis, Ingham--M\"untz inequality.

~\newline\textbf{The present work is the result of the split of \cite{lhachemi2025controllability} into the study of the controllability properties of the cascade reported in this paper and the explicit feedback stabilization problem studied in the companion paper \cite{lhachemi2025tac}.}

\section{Introduction}\label{sec:intro}

Coupled PDEs involving different propagation regimes arise in many applications (thermoelasticity, fluid--structure interactions, or more generally multi-physics models).
From a control-theoretic viewpoint, hyperbolic--parabolic couplings are particularly delicate because they mix finite-speed propagation and parabolic smoothing.
The present paper is devoted to a detailed \emph{spectral} analysis of boundary controllability for two one-dimensional cascade configurations coupling a wave equation and a reaction--diffusion equation.

\smallskip

\paragraph{Wave--heat cascade}
Our first and main model is the wave--heat cascade
\begin{subequations}\label{eq:WH}
    \begin{align}
        &\partial_t y(t,x) = \partial_{xx} y(t,x) + c\, y(t,x) + \beta(x)\, z(t,x) , \label{eq:WH_y} \\
        &\partial_{tt} z(t,x) = \partial_{xx} z(t,x) , \label{eq:WH_z} \\
        & y(t,0) = y(t,L) = 0 , \qquad z(t,0) = 0 , \ \ \partial_x z(t,L) = u(t) , \label{eq:WH_bc}
    \end{align}
\end{subequations}
where $(t,x)\in(0,T)\times (0,L)$ with $T>0$ and $L>0$, $c\in\mathbb{R}$ and $\beta\in L^\infty(0,L)$ is a real-valued function.
The boundary control $u\in L^2(0,T)$ acts on the Neumann trace of the wave component.
The wave dynamics influences the heat dynamics through the internal coupling term $\beta z$ in \eqref{eq:WH_y}.
This model is the controllability counterpart of the stabilization problem studied in the companion paper \cite{lhachemi2025tac}, where an explicit feedback is designed and analyzed.

\smallskip

\paragraph{Heat--wave cascade}
We also consider the reversed configuration
\begin{subequations}\label{eq:HW}
\begin{align}
&\partial_t y(t,x) = \partial_{xx}y(t,x) + c\,y(t,x) ,  \label{eq:HW_y}\\
&\partial_{tt} z(t,x) = \partial_{xx}z(t,x) + \beta(x)\,y(t,x) ,  \label{eq:HW_z}\\
&y(t,0) = 0,\ \ y(t,L) = u(t), \qquad z(t,0) = 0,\ \ \partial_x z(t,L) = 0 ,  \label{eq:HW_bc}
\end{align}
\end{subequations}
in which the boundary control acts on the heat component (Dirichlet trace) and drives the wave component only through the coupling term $\beta y$ in \eqref{eq:HW_z}.

\smallskip

\paragraph{Related literature}
The controllability and stabilization of coupled hyperbolic--parabolic systems have been studied in various contexts; we only mention a few references closely connected to the present approach.
The works \cite{zhang2003polynomial,zhang2004polynomial} on boundary-coupled heat--wave systems  pioneered the use of combined Ingham--M\"untz inequalities to capture the interaction between a hyperbolic spectrum and a parabolic spectrum.
The general functional-analytic framework of well-posed linear systems and admissible control/observation operators is discussed in, e.g., \cite{trelat_SB,tucsnakweiss}, and classical references on controllability for PDEs include \cite{lions,russell1978controllability}.

\smallskip

\paragraph{Contributions and novelty}
Our results provide a \emph{complete} characterization of exact controllability for \eqref{eq:WH} and a parallel analysis for \eqref{eq:HW}.
The main outputs for \eqref{eq:WH} can be summarized as follows:

\begin{itemize}
\item We identify the sharp minimal control time imposed by the wave component: $T>2L$ (the critical case $T=2L$ remains open for the coupled problem).
\item We show that exact controllability hinges on a sequence of coupling coefficients $(\gamma_n)_{n\geq1}$ obtained from the interaction between the coupling profile $\beta$ and the wave/heat eigenmodes.
Under a coupling condition $\gamma_n\neq0$ for all $n$, we derive an observability inequality for the adjoint dynamics and deduce exact controllability in a weighted space $V$ (and exact \emph{null} controllability in a larger space $V_0$).
\item The controllability space $V$ is defined in a spectral way by coupling-dependent weights that may be highly oscillating and thus $V$ does not coincide with any standard Sobolev space in general. Controllability is generic (full measure) in natural parameter families but may fail to be robust.
\item We prove that $V$ is not invariant along HUM trajectories: even if the initial and final states are in $V$, the minimal-energy trajectory may leave $V$ for intermediate times.
\end{itemize}
For the heat--wave cascade \eqref{eq:HW}, we obtain an analogous characterization, and we emphasize the qualitative difference between the two directions of coupling.

\smallskip

\paragraph{Organization}
Section~\ref{sec:wave_heat} analyzes the wave--heat cascade \eqref{eq:WH}: Section~\ref{sec:spectral_WHT} recalls the spectral setting, and Section~\ref{sec:controllability_WHT} establishes the observability/controllability results and discusses consequences (including the HUM non-invariance phenomenon).
Section~\ref{sec:heat_wave} treats the heat--wave cascade \eqref{eq:HW} and compares it with the wave--heat configuration.
Section~\ref{sec_conclusion} is a conclusion with open problems and perspectives.

Technical points (transposition definition of the control operator, admissibility, and the Ingham--M\"untz inequality) are gathered in the appendices, together with the detailed proof of the HUM non-invariance statement.

\section{Controllability of a wave-heat cascade}\label{sec:wave_heat}
We consider the system \eqref{eq:WH} for $t>0$ and $x\in(0,L)$, where $u\in L^2(0,T)$ is the control input.
Although the model is real‑valued, we work in the complexification for spectral arguments.
For an initial state $(y_0,z_0,z_1)$, we denote by 
$$
\mathcal{X}(t)=(y(t,\cdot),z(t,\cdot),\partial_t z(t,\cdot))^\top
$$
the corresponding state at time $t$.
Throughout the paper, we work in the \emph{complexified} energy space 
$$
\mathcal{H} = L^2(0,L;\mathbb{C})\times H^1_{(0)}(0,L;\mathbb{C})\times L^2(0,L;\mathbb{C}), 
$$
which will be used as a pivot Hilbert space, endowed with the inner product
$$
\langle (f_1,g_1,h_1),(f_2,g_2,h_2)\rangle_{\mathcal{H}}
= \int_0^L \big(f_1 \overline{f_2} + g_1' \overline{g_2'} + h_1 \overline{h_2}\big)\,\mathrm{d}x ,
$$
with 
$$
H^1_{(0)}(0,L;\mathbb{C}) = \{g\in H^1(0,L;\mathbb{C})\ \mid\ g(0)=0\} .
$$
When the context is clear, we simply write $\langle\cdot,\cdot\rangle$ for $\langle\cdot,\cdot\rangle_{\mathcal{H}}$ and $\|\cdot\|_{\mathcal{H}}$ for the associated norm.
In what follows, for brevity, all function spaces are complex unless explicitly stated otherwise.

Since $\mathcal{X}=(y,z,\partial_t z)^\top$, the control system \eqref{eq:WH} can be written as
\begin{equation}\label{eq:abstract_WHT}
\dot{\mathcal{X}}(t)=\mathcal{A} \mathcal{X}(t)+\mathcal{B} u(t),
\end{equation}
where $\mathcal{A}:D(\mathcal{A})\subset\mathcal{H}\to\mathcal{H}$ is defined by
$$
\mathcal{A} = \begin{pmatrix} \partial_{xx} + c\,\mathrm{id} & \beta\,\mathrm{id} & 0 \\ 0 & 0 & \mathrm{id} \\ 0 & \partial_{xx} & 0 \end{pmatrix}
$$
i.e., $\mathcal{A}(f,g,h)=\big(f''+cf+\beta g,\, h,\, g''\big)$, with domain
\begin{multline*}
D(\mathcal{A}) = \{  (f,g,h)\in H^2(0,L) \times H^2(0,L) \times H^1(0,L)\ \mid\ \\
f(0)=f(L)=g(0)=g'(L)=h(0)=0 \} .
\end{multline*}
The (unbounded) control operator $\mathcal{B}$ associated with the Neumann boundary control on the wave equation can be defined by transposition; see Appendix~\ref{sec_app_A}.
In particular,
\begin{equation}\label{eq:B0star}
\mathcal{B}^*(\psi_1,\psi_2,\psi_3)^\top=\psi_3(L)\qquad \forall(\psi_1,\psi_2,\psi_3)^\top\in D(\mathcal{A}^*).
\end{equation}
Well-posedness (admissibility of $\mathcal{B}$) is shown in Appendix~\ref{sec_app_B}.

\begin{remark}[Real-valued trajectories]\label{rem:real_trajectories}
Let
$$
\mathcal{H}_{\mathbb{R}} = L^2(0,L;\mathbb{R})\times H^1_{(0)}(0,L;\mathbb{R})\times L^2(0,L;\mathbb{R}),
$$
viewed as a closed real subspace of $\mathcal{H}$.
Since $c$ and $\beta$ are real, the operators $\mathcal{A}$ and $\mathcal{B}$ commute with complex conjugation.
Hence, for any real control $u\in L^2(0,T;\mathbb{R})$ and any real initial condition $\mathcal{X}(0)\in\mathcal{H}_{\mathbb{R}}$, the corresponding mild solution satisfies $\mathcal{X}(t)\in\mathcal{H}_{\mathbb{R}}$ for all $t\in[0,T]$.

Moreover, controllability statements proved below in the complex setting restrict to the real one.
Indeed, if $\mathcal{X}_0,\mathcal{X}_1\in \mathcal{H}_{\mathbb{R}}$ and $u\in L^2(0,T;\mathbb{C})$ steers $\mathcal{X}_0$ to $\mathcal{X}_1$, then $\mathrm{Re}(u)$ steers $\mathcal{X}_0$ to $\mathcal{X}_1$ as well (real and imaginary parts satisfy decoupled dynamics).
\end{remark}

\subsection{Spectral properties of $\mathcal{A}$ and $\mathcal{A}^*$}\label{sec:spectral_WHT}
The spectral material of this section is shared with \cite{lhachemi2025tac}; we reproduce the statements needed for controllability and keep the proofs to the strict minimum.

\begin{lemma}\label{lem: eigenstructures of A0}
The eigenvalues of $\mathcal{A}$ split into a parabolic spectrum
\begin{equation*}
\lambda_{1,n}=c-\frac{n^2\pi^2}{L^2},\qquad n\in\mathbb{N}^*,
\end{equation*}
and a hyperbolic spectrum
\begin{equation*}
\lambda_{2,m}= i\,\frac{(2m+1)\pi}{2L},\qquad m\in\mathbb{Z}.
\end{equation*}
Associated eigenvectors are given by:
\begin{enumerate}
\item\label{item_parabolic} for each $n\in\mathbb{N}^*$, $\phi_{1,n}=\big(\phi_{1,n}^1,0,0\big)$ with
$\phi_{1,n}^1(x)=\sqrt{\frac{2}{L}}\sin\!\big(\frac{n\pi}{L}x\big)$ for all $x\in(0,L)$;
\item\label{item_hyperbolic} for each $m\in\mathbb{Z}$
$\phi_{2,m}=\big(\phi_{2,m}^1,\phi_{2,m}^2,\phi_{2,m}^3\big)$ with
$$
\phi_{2,m}^2(x)=\tfrac{2\sqrt{L}}{|2m+1|\pi}\sinh(\lambda_{2,m}x),\qquad \phi_{2,m}^3(x)=\lambda_{2,m}\phi_{2,m}^2(x),
$$
for all $x\in(0,L)$,
where $\phi_{2,m}^1\in H^2(0,L)$ is the unique solution of the differential equation
$(\lambda_{2,m}-c)\,\phi_{2,m}^1(x)-(\phi_{2,m}^1)''(x)=\beta(x)\,\phi_{2,m}^2(x)$ in $(0,L)$ with
$\phi_{2,m}^1(0)=\phi_{2,m}^1(L)=0$.
\end{enumerate}
\end{lemma}

\begin{proof}[Sketch of proof]
Item~\ref{item_parabolic} follows from the decoupled heat operator with homogeneous Dirichlet conditions.
For Item~\ref{item_hyperbolic}, solving $\lambda^2 g=g''$ with $g(0)=0$, $g'(L)=0$ yields $\lambda=\lambda_{2,m}$ and $g(x)=\mathrm{Cst}\,\sinh(\lambda_{2,m}x)$.
Then $\phi_{2,m}^1$ is obtained by solving $(\lambda_{2,m}-\partial_{xx}-c)\phi_{2,m}^1=\beta \phi_{2,m}^2$ with Dirichlet conditions.
See \cite{lhachemi2025tac} for full details.
\end{proof}

\begin{remark}\label{rem_spectrum}
The cascade nature of \eqref{eq:WH} is reflected by the nonzero first component $\phi_{2,m}^1$ (whose explicit expression can be found in \cite{lhachemi2025tac}).
\end{remark}

\begin{lemma}\label{lem:A0 is Riesz spectral}
The family $\Phi=\{\phi_{1,n}\}_{n\in\mathbb{N}^*}\cup\{\phi_{2,m}\}_{m\in\mathbb{Z}}$ is a Riesz basis of $\mathcal{H}$.
In particular, $\mathcal{A}$ is a Riesz-spectral operator: there exists a biorthogonal family $\Psi=\{\psi_{1,n}\}_{n\in\mathbb{N}^*}\cup\{\psi_{2,m}\}_{m\in\mathbb{Z}}\subset \mathcal{H}$ such that
$\langle \phi_{i,k},\psi_{j,\ell}\rangle=\delta_{ij}\delta_{k\ell}$
and the semigroup generated by $\mathcal{A}$ admits the spectral expansion
\begin{equation}\label{eq: semigroup expansion}
T_0(t)\mathcal{X}=\sum_{n\in\mathbb{N}^*}e^{\lambda_{1,n}t}\langle\mathcal{X},\psi_{1,n}\rangle\phi_{1,n}
+\sum_{m\in\mathbb{Z}}e^{\lambda_{2,m}t}\langle\mathcal{X},\psi_{2,m}\rangle\phi_{2,m}
\qquad \forall t\geq0.
\end{equation}
\end{lemma}

\begin{proof}[Sketch of proof]
This follows from a perturbation argument comparing $\Phi$ with the decoupled Riesz basis corresponding to $\beta=0$, and from Bari's theorem \cite{gohberg1978introduction}.
We refer to \cite{lhachemi2025tac} for the detailed estimates (in particular, the $\ell^2$-summability of the coupling correction in $\phi_{2,m}^1$).
\end{proof}

\smallskip

Exact controllability is formulated via observability of the adjoint dynamics.
We therefore record the adjoint operator and its eigenvectors.

\begin{lemma}\label{lem: adjoint operator A0*}
The adjoint $\mathcal{A}^*:D(\mathcal{A}^*)\subset\mathcal{H}\to\mathcal{H}$ is given by
$$
\mathcal{A}^* = \begin{pmatrix} \partial_{xx} + c\,\mathrm{id} & 0 & 0 \\ P_\beta & 0 & - \mathrm{id} \\ 0 & -\partial_{xx} & 0 \end{pmatrix}
$$
i.e., $\mathcal{A}^*(f,g,h)=\big(f''+cf,\, P_\beta f-h,\, -g''\big)$, with domain
\begin{multline*}
D(\mathcal{A}^*) = \{ (f,g,h)\in H^2(0,L) \times H^2(0,L) \times H^1(0,L)\ \mid\ \\
f(0)=f(L)=g(0)=g'(L)=h(0)=0 \} .
\end{multline*}
where 
$P_\beta f(x) = \int_0^x \int_\tau^L \beta(s) f(s) \,\mathrm{d}s\,\mathrm{d}\tau$ for every $x\in[0,L]$;
equivalently, $w=P_\beta f$ is the unique solution of $w''=-\beta f$ on $(0,L)$ with $w(0)=0$ and $w'(L)=0$.

The eigenvalues of $\mathcal{A}^*$ are $\lambda_{1,n}$ and $\overline{\lambda_{2,m}}$.
Its eigenfunctions are given by the dual Riesz basis $\Psi=\{\psi_{1,n}\}_{n\in\mathbb{N}^*}\cup\{\psi_{2,m}\}_{m\in\mathbb{Z}}$ of $\Phi$, with $\psi_{2,m}=(0,\psi_{2,m}^2,\psi_{2,m}^3)$ and 
\begin{equation}\label{psi_2m}
\psi_{2,m}^2(x)=\tfrac{2\sqrt{L}}{|2m+1|\pi}\sinh(\overline{\lambda_{2,m}}x),\qquad
\psi_{2,m}^3(x)=-\overline{\lambda_{2,m}}\,\psi_{2,m}^2(x),
\end{equation}
and $\psi_{1,n}=(\psi_{1,n}^1,\psi_{1,n}^2,\psi_{1,n}^3)$ with $\psi_{1,n}^1(x)=\sqrt{\tfrac{2}{L}}\sin(\tfrac{n\pi}{L}x)$ and
\begin{subequations}\label{psi_1n}
\begin{align}
\psi_{1,n}^2(x)&=-\tfrac{\gamma_n}{\lambda_{1,n}^2\cosh(\lambda_{1,n}L)}\sqrt{\tfrac{2}{L}}\cosh(\lambda_{1,n}(x-L)) \label{psi_1n_2} \\
&\qquad -\tfrac{1}{\lambda_{1,n}^2}\sqrt{\tfrac{2}{L}}\int_x^L \beta(s)\sin\!\left(\tfrac{n\pi}{L}s\right)\sinh(\lambda_{1,n}(x-s))\,\mathrm{d}s \notag \\
&\qquad +\tfrac{1}{\lambda_{1,n}}\sqrt{\tfrac{2}{L}}\int_0^x\int_\tau^L \beta(s)\sin\!\left(\tfrac{n\pi}{L}s\right)\,\mathrm{d}s\,\mathrm{d}\tau, \notag \\
\psi_{1,n}^3(x)&=\tfrac{\gamma_n}{\lambda_{1,n}\cosh(\lambda_{1,n}L)}\sqrt{\tfrac{2}{L}}\cosh(\lambda_{1,n}(x-L)) \label{psi_1n_3} \\
&\qquad +\tfrac{1}{\lambda_{1,n}}\sqrt{\tfrac{2}{L}}\int_x^L \beta(s)\sin\!\left(\tfrac{n\pi}{L}s\right)\sinh(\lambda_{1,n}(x-s))\,\mathrm{d}s, \notag
\end{align}
\end{subequations}
if $\lambda_{1,n}\neq0$, and
$$
\psi^2_{1,n}(x) =  0 , \qquad 
\psi^3_{1,n}(x) = \sqrt{\tfrac{2}{L}} \int_0^x \int_\tau^L \beta(s) \sin\left(\tfrac{n\pi}{L}s\right) \,\mathrm{d}s\,\mathrm{d}\tau 
$$
if $\lambda_{1,n}=0$.
The coupling coefficients $\gamma_n$ are defined by
\begin{equation}\label{def_gamma_n}
\gamma_n=\left\{
\begin{array}{lcl}
\displaystyle\int_0^L \beta(s)\sin\!\left(\tfrac{n\pi}{L}s\right)\sinh(\lambda_{1,n}s)\,\mathrm{d}s & \text{if} & \lambda_{1,n}\neq0,\\[3mm]
\displaystyle \int_0^L \beta(s)\sin\!\left(\tfrac{n\pi}{L}s\right)\, s\,\mathrm{d}s & \text{if} & \lambda_{1,n}=0.
\end{array}\right.
\end{equation}
\end{lemma}

\begin{proof}[Sketch of proof]
The expression of $\mathcal{A}^*$ follows from integration by parts in $\langle \mathcal{A}\cdot,\cdot\rangle_{\mathcal{H}}$.
Eigenvectors are obtained by solving $\mathcal{A}^*\psi=\lambda\psi$ componentwise and using the boundary conditions in $D(\mathcal{A}^*)$.
We omit the straightforward calculations and refer to \cite{lhachemi2025tac}.
\end{proof}

\begin{remark}\label{remgamman}
The coefficients $\gamma_n$ defined by \eqref{def_gamma_n} play a pivotal role in the controllability of \eqref{eq:WH}.
They encode the interaction between the internal coupling profile $\beta$ and the wave eigenmodes through the strongly growing factor $\sinh(\lambda_{1,n}s)$.
It can be noted that $\vert\gamma_n\vert\leq \frac{\mathrm{Cst}}{n^2}e^{n^2\pi^2/L}$.
\end{remark}

\subsection{Exact controllability and observability properties}\label{sec:controllability_WHT}
\subsubsection{Preliminaries}
We establish exact and null controllability in suitable Hilbert spaces by duality.
Let $T>0$.
Exact observability for the adjoint system
\begin{equation}\label{eq:adjoint_WHT}
\dot{\mathcal{X}}(t)=\mathcal{A}^*\mathcal{X}(t)
\end{equation}
with observation $\mathcal{B}^*\mathcal{X}(t)=\mathcal{X}^3(t,L)$ (see \eqref{eq:B0star}), reads:
there exists $C_T>0$ such that
\begin{equation}\label{exact_obs}
\int_0^T|\mathcal{X}^3(t,L)|^2\,\mathrm{d}t\geq C_T\|\mathcal{X}(0)\|_{V'}^2,
\end{equation}
for all initial data $\mathcal{X}(0)$ in a suitable Hilbert space $V'\subset \mathcal{H}$ (to be defined).
Similarly, finite-time observability reads
\begin{equation}\label{exact_obs1}
\int_0^T|\mathcal{X}^3(t,L)|^2\,\mathrm{d}t\geq C_T^0\|\mathcal{X}(T)\|_{\mathcal{H}}^2.
\end{equation}
By standard HUM duality (see \cite{lions,trelat_SB,tucsnakweiss}), these inequalities yield exact controllability in $V$ and exact null controllability in $V_0$ (defined below).

\smallskip

A first structural condition is the non-vanishing of the modal coefficients $\gamma_n$.

\begin{lemma}\label{lem_necessary}
A necessary condition for \eqref{eq:WH} to be approximately (and hence exactly) controllable in any Hilbert space continuously and densely embedded into $\mathcal{H}$ is that $\gamma_n\neq 0$ for every $n\in\mathbb{N}^*$, where $\gamma_n$ is defined by \eqref{def_gamma_n}.
\end{lemma}

\begin{proof}
Let $n\in\mathbb{N}^*$.
If $\gamma_n=0$, then by Lemma~\ref{lem: adjoint operator A0*} we have $\mathcal{B}^*\psi_{1,n}=\psi_{1,n}^3(L)=0$.
Therefore $\psi_{1,n}$ is an unobservable eigenvector of $\mathcal{A}^*$, so that neither approximate observability in $\mathcal{H}$ nor approximate controllability (in any compatible Hilbert space) can hold.
\end{proof}

The next lemma provides the asymptotics of the observation coefficients.

\begin{lemma}\label{lem_necessary_psi}
For any $m\in\mathbb{Z}$, one has $|\psi_{2,m}^3(L)|=1/\sqrt{L}$.
Moreover, for $n\to+\infty$,
\begin{equation}\label{eq:psi_1n_asympt}
\psi_{1,n}^3(L)\sim -\frac{(2L)^{3/2}e^{cL}}{\pi^2}\,\frac{\gamma_n}{n^2}\,e^{-n^2\pi^2/L}.
\end{equation}
\end{lemma}

\begin{proof}
Using \eqref{psi_2m} and $|\sinh(\overline{\lambda_{2,m}}L)|=1$, we get $|\psi_{2,m}^3(L)|=1/\sqrt{L}$.
For \eqref{eq:psi_1n_asympt}, one uses the explicit formula \eqref{psi_1n_3} and the growth of $\cosh(\lambda_{1,n}L)$ as $n\to\infty$.
\end{proof}

We now define the observability and controllability spaces.
Set
\begin{equation}\label{def_nu}
\nu=\frac{2\pi^2}{L}\Big(1+\frac{T}{L}\Big).
\end{equation}
Assume \eqref{assumption_gamma_n}.
Define
\begin{multline}\label{def_V'}
V' = \bigg\{ \sum_{n\in\mathbb{N}^*} a_n\psi_{1,n}+\sum_{m\in\mathbb{Z}} b_m\psi_{2,m}\quad \Big|\quad a_n\in\mathbb{C},\ b_m\in\mathbb{C},  \\
\sum_{n\geq1}\frac{\gamma_n^2}{n^4}e^{-\nu n^2}|a_n|^2+\sum_{m\in\mathbb{Z}}|b_m|^2<\infty  \bigg\},
\end{multline}
endowed with
\begin{equation}\label{def_norm_V'}
\|\mathcal{X}\|_{V'}^2=\sum_{n\geq1}\frac{\gamma_n^2}{n^4}e^{-\nu n^2}\,|\langle\mathcal{X},\phi_{1,n}\rangle|^2+\sum_{m\in\mathbb{Z}}|\langle\mathcal{X},\phi_{2,m}\rangle|^2.
\end{equation}
Let $V$ be the dual of $V'$ with respect to the pivot space $\mathcal{H}$, i.e.
\begin{multline}\label{def_V}
V = \bigg\{ \sum_{n\geq1} a_n\phi_{1,n}+\sum_{m\in\mathbb{Z}} b_m\phi_{2,m}\quad \Big|\quad a_n\in\mathbb{C},\ b_m\in\mathbb{C},  \\
\sum_{n\geq1}\frac{n^4}{\gamma_n^2}e^{\nu n^2}|a_n|^2+\sum_{m\in\mathbb{Z}}|b_m|^2<\infty \bigg\} ,
\end{multline}
with
\begin{equation}\label{def_norm_V}
\|\mathcal{X}\|_{V}^2 = \sum_{n\geq1}\frac{n^4}{\gamma_n^2}e^{\nu n^2}\,|\langle\mathcal{X},\psi_{1,n}\rangle|^2+\sum_{m\in\mathbb{Z}}|\langle\mathcal{X},\psi_{2,m}\rangle|^2.
\end{equation}
Finally, the Hilbert spaces $V_0$ and $V_0'$ are defined similarly as $V$ and $V'$ but with a different $\nu$, namely, $\nu = \frac{2\pi^2}{L}$ (in contrast to \eqref{def_nu}, here $\nu$ does not depend on $T$).

\smallskip

\begin{lemma}\label{lem:embeddings_V}
We have $\{0\}\times H_{(0)}^1(0,L) \times L^2(0,L)\subset V\subset V_0$ and the continuous embeddings $V\subset V_0\subset \mathcal{H}\subset V_0'\subset V'$ hold with dense injections.
\end{lemma}

\begin{proof}
The claims easily follow from the definitions of the spaces, using that $\vert\gamma_n\vert\leq \frac{\mathrm{Cst}}{n^2}e^{n^2\pi^2/L}$ (see Remark \ref{remgamman}).
\end{proof}

\subsubsection{Main result on the wave-heat cascade}

\begin{theorem}\label{thm_main}
Assume that $T>2L$ and that 
\begin{equation}\label{assumption_gamma_n}
\gamma_n\neq 0\qquad\forall n\in\mathbb{N}^*,
\end{equation}
where $\gamma_n$ is defined by \eqref{def_gamma_n}. Then:
\begin{enumerate}
\item\label{item_obs} There exists $C_T>0$ such that the observability inequality \eqref{exact_obs} is satisfied.
\item\label{item_finitetimeobs} There exists $C_T^0>0$ such that the finite-time observability inequality \eqref{exact_obs1} is satisfied.
\item\label{item_exact} The control system \eqref{eq:WH} is exactly controllable in time $T$ in the space $V$, i.e., given any $\mathcal{X}_0,\mathcal{X}_1\in V$, there exists $u \in L^2(0,T)$ such that the solution $\mathcal{X}=(y,z,\partial_t z)^\top$ of \eqref{eq:WH} starting at $\mathcal{X}(0)=\mathcal{X}_0$ satisfies $\mathcal{X}(T)=\mathcal{X}_1$.
\item\label{item_exact_null} The control system \eqref{eq:WH} is exactly null controllable in time $T$ in the space $V_0$, i.e., given any $\mathcal{X}_0\in V_0$, there exists $u \in L^2(0,T)$ such that the solution $\mathcal{X}=(y,z,\partial_t z)^\top$ of \eqref{eq:WH} starting at $\mathcal{X}(0)=\mathcal{X}_0$ satisfies $\mathcal{X}(T)=0$.
\item\label{item_approx} The control system \eqref{eq:WH} is approximately controllable in time $T$ in the Hilbert space $\mathcal{H}$,
i.e., given any $\mathcal{X}_0,\mathcal{X}_1\in \mathcal{H}$ and any $\varepsilon>0$, there exists $u \in L^2(0,T)$ such that the solution $\mathcal{X}=(y,z,\partial_t z)^\top$ of \eqref{eq:WH} starting at $\mathcal{X}(0)=\mathcal{X}_0$ satisfies $\Vert \mathcal{X}(T)-\mathcal{X}_1\Vert_{\mathcal{H}}\leq\varepsilon$ .
\end{enumerate}
Moreover:%
\begin{itemize}
\item If $T<2L$, then \eqref{eq:WH} is 
not exactly controllable, not exactly null controllable, and not approximately controllable
in time $T$ in any Hilbert space continuously and densely embedded in $\mathcal{H}$.
\item If \eqref{assumption_gamma_n} fails (i.e., $\gamma_{n_0}=0$ for some $n_0$), then \eqref{eq:WH} is
not exactly controllable, not exactly null controllable, and not approximately controllable
for any time $T>0$ in any Hilbert space continuously and densely embedded in $\mathcal{H}$.
\end{itemize}
\end{theorem}

In Theorem \ref{thm_main}, as already written in Remark~\ref{rem:real_trajectories}, the real-valued case follows by restriction to real initial/final data and real controls.


\begin{proof}
By (the dual of) Lemma~\ref{lem:A0 is Riesz spectral}, any solution of \eqref{eq:adjoint_WHT} admits the expansion
$$
\mathcal{X}(t)=\sum_{n\geq1}e^{\lambda_{1,n}t}\langle \mathcal{X}(0),\phi_{1,n}\rangle \psi_{1,n}
+\sum_{m\in\mathbb{Z}}e^{\overline{\lambda_{2,m}}t}\langle \mathcal{X}(0),\phi_{2,m}\rangle \psi_{2,m}.
$$
Using $\mathcal{B}^*\mathcal{X}(t)=\mathcal{X}^3(t,L)$ and Lemma~\ref{lem: adjoint operator A0*}, we obtain
$$
\mathcal{X}^3(t,L)=\sum_{n\geq1}e^{\lambda_{1,n}t}\langle \mathcal{X}(0),\phi_{1,n}\rangle \psi_{1,n}^3(L)
+\sum_{m\in\mathbb{Z}}e^{\overline{\lambda_{2,m}}t}\langle \mathcal{X}(0),\phi_{2,m}\rangle \psi_{2,m}^3(L).
$$
Applying the Ingham--M\"untz inequality (Appendix~\ref{sec_IM}) to the sequences $(\lambda_{1,n})_{n\geq1}$ and $(\overline{\lambda_{2,m}})_{m\in\mathbb{Z}}$, we infer that for any $T>2L$ there exists $C>0$ such that
\begin{multline*}
\int_0^T|\mathcal{X}^3(t,L)|^2\,\mathrm{d}t 
\geq C\bigg(\sum_{n\geq1}e^{2\lambda_{1,n}T}|\langle \mathcal{X}(0),\phi_{1,n}\rangle|^2|\psi_{1,n}^3(L)|^2 \\
+\sum_{m\in\mathbb{Z}}|\langle \mathcal{X}(0),\phi_{2,m}\rangle|^2|\psi_{2,m}^3(L)|^2\bigg).
\end{multline*}
Using Lemma~\ref{lem_necessary_psi} and the definition of $\nu$ in \eqref{def_nu}, we recover \eqref{exact_obs} with the norm \eqref{def_norm_V'}. This proves Item~\ref{item_obs}.
Item~\ref{item_finitetimeobs} is proved similarly.

By classical duality arguments due to \cite{lions} (see, e.g., \cite[Theorem 5.30]{trelat_SB}, \cite[Theorem 11.2.1]{tucsnakweiss}), exact observability in $V'$ is equivalent to exact controllability in the dual space $V$, and finite-time observability \eqref{exact_obs1} yields null controllability in $V_0$.
This gives Items \ref{item_exact} and \ref{item_exact_null}.

By density, Item~\ref{item_exact} implies Item~\ref{item_approx}.

If $\gamma_{n_0}=0$ for some $n_0$, Lemma~\ref{lem_necessary} shows that approximate controllability fails for any $T$ (actually, the 1D subspace $\mathbb{C}\,\psi_{1,n_0}$ is not in the closure of the controllability space).
If $T<2L$, the wave subsystem in \eqref{eq:WH_z} cannot be approximately controlled in time $T$ with a single boundary control; consequently, the coupled system cannot be approximately controlled in time $T$ in any compatible Hilbert space.
\end{proof}

Several remarks are in order.

\begin{remark}[On the time condition]\label{rem:time}
The threshold $T>2L$ is dictated by the hyperbolic spectrum $\lambda_{2,m}$ and coincides with the minimal time for boundary observability (equivalently, exact controllability) of the 1D wave equation with the mixed boundary conditions $z(t,0)=0$ and $\partial_x z(t,L)=0$.
For $T=2L$, the wave equation is still exactly controllable, but the Ingham--M\"untz inequality used in the present proof does not cover this limit case. Therefore, whether Theorem~\ref{thm_main} remains valid for the coupled cascade when $T=2L$ is left open.

For $T<2L$, the wave equation is not exactly observable, but one can still derive refined \emph{regional} or \emph{partial} observability estimates describing which components of the initial data are seen in time $T$ (see \cite{dehman2025regional}).
Obtaining analogous descriptions for the coupled cascades when $T\leq2L$ is an interesting open direction.
\end{remark}

\begin{remark}[Consequence: stabilization]\label{rem:stabilization}
As a consequence of exact null controllability of \eqref{eq:WH} in time $T$ in the space $V_0$, one can deduce exponential stabilizability in $V_0$ of the wave-heat cascade, at any prescribed decay rate, by means of a bounded feedback (see \cite{liuwangxuyu,trelatwangxu}).

This implication is however mainly qualitative: the feedback provided by abstract controllability--stabilization results is typically defined through a controllability Gramian (or an operator Riccati equation) and does not yield an explicit implementable formula.
Moreover, the resulting stabilizing operator may depend very sensitively on the coupling profile $\beta$, consistently with the non-robustness phenomena highlighted in Remark~\ref{rem_nonrobust} further.
In particular, one should not expect such abstract feedbacks to be robust under realistic uncertainties on $\beta$.

Constructing explicit robust boundary feedback laws for \eqref{eq:WH} is precisely the objective of \cite{lhachemi2025tac}.
\end{remark}

\begin{remark}
Exact controllability in $V$ implies exact null controllability in $V$.
However, Item~\ref{item_exact_null} of Theorem~\ref{thm_main} yields the stronger property of exact null controllability in the larger space $V_0\supset V$ (see Lemma~\ref{lem:embeddings_V}).
\end{remark}

\begin{remark}[On the controllability spaces]\label{rem:spaceV}
The spectral definition of $V$ and $V'$ follows the strategy introduced in \cite{zhang2004polynomial} (see also \cite{zhang2003polynomial}), where an Ingham--M\"untz inequality is used to handle simultaneously a hyperbolic family of purely imaginary eigenvalues and a parabolic family with rapidly decaying real parts.

In \cite{zhang2004polynomial}, the coupling between the heat and the wave equations is applied at the boundary and the resulting controllability weights are monotone (heuristically, as if $\gamma_n\equiv1$ in \eqref{def_V}).
In contrast, for the present \emph{internal} coupling, the coefficients $\gamma_n$ in \eqref{def_gamma_n} depend on the full profile of $\beta$ and may exhibit strong oscillations and cancellations.
As a consequence, the exact and null controllability spaces $V$ and $V_0$ need not admit a simple characterization in standard Sobolev scales: their norms involve the non-monotone weights $n^4e^{\nu n^2}/|\gamma_n|^2$, which may 
oscillate 
between distinct exponential regimes.
This nonstandard feature is central in our analysis and motivates the examples and genericity discussion in Section \ref{sec_examples_gamma_n} further, where we estimate $\gamma_n$ for specific classes of coupling profiles.
\end{remark}

\begin{remark}\label{rem:decomp}
The definition \eqref{def_V} of $V$ and the corresponding definition of $V_0$ show that, compared with the energy space $\mathcal{H}$, the constraints imposed by exact/null controllability are essentially carried by the parabolic coefficients $a_n$ through the weights involving $\gamma_n$ and the exponential factors.
In contrast, in the hyperbolic coefficients $b_m$, the spaces $V$, $V_0$ coincide with the standard wave energy scale (see also Lemma \ref{lem:embeddings_V}).
This is consistent with the fact that when $\mathcal{X}^1\equiv0$ the observability inequality reduces to the classical boundary observability of the wave equation.
\end{remark}

\begin{remark}[Dirichlet actuation on the wave component]
In \eqref{eq:WH} we have chosen a Neumann boundary control on the wave equation, i.e., $\partial_x z(t,L)=u(t)$, which corresponds to the observation operator $\mathcal{B}^*\psi=\psi_3(L)$ in \eqref{eq:B0star} and to hyperbolic modes with frequencies $\omega_m=\frac{(2m+1)\pi}{2L}$.

If one prescribes a Dirichlet control $z(t,L)=u(t)$ (with $z(t,0)=0$) in \eqref{eq:WH_bc} instead of a Neumann control, 
the natural pivot space becomes
$$
\mathcal{H}_{\mathrm{D}}=L^2(0,L)\times H_0^1(0,L)\times L^2(0,L)
$$
and the hyperbolic eigenfunctions are $\sin\!\big(\frac{n\pi x}{L}\big)$ with frequencies $\omega_n=\frac{n\pi}{L}$. In the dual system, the boundary observation becomes a Neumann trace of the adjoint wave displacement, namely $\mathcal{B}_{\mathrm{D}}^*\psi = -\partial_x \psi_2(L)$,
so that on the hyperbolic eigenmodes one has that $|\mathcal{B}_{\mathrm{D}}^*\psi_{2,n}|$ 
is of the same order as $\omega_n$, as $n$ inscreases.

As a consequence, the hyperbolic part of the observability norm acquires an additional factor $\omega_n^2$, and the corresponding hyperbolic weights in the controllability space are weakened by a factor $\omega_n^{-2}$ (compared with the Neumann-controlled case), which enlarge the hyperbolic component of the controllability space (in the sense that the associated spectral weights in $V$ are weaker than in the Neumann-controlled case).
The parabolic weights (involving $\gamma_n$) are unchanged.
\end{remark}

\begin{remark}[Almost optimality of the exponential weights]\label{rem:almost_opt}
Under \eqref{assumption_gamma_n}, the spaces $V$ and $V_0$ are close to being the largest subspaces of $\mathcal{H}$ in which \eqref{eq:WH} can be exactly/null controllable with controls $u\in L^2(0,T)$.
Indeed, the parabolic (M\"untz) part of the Ingham--M\"untz inequality is known to be almost sharp, which makes the exponential weights in \eqref{def_V} essentially unavoidable (see Appendix~\ref{sec_IM}).
\end{remark}

\subsubsection{Non-invariance of $V$ along HUM trajectories}
In Item~\ref{item_exact} (resp., in Item~\ref{item_exact_null}), although the initial condition is taken in the smaller space $V\subset\mathcal{H}$ (resp., $V_0\subset\mathcal{H}$), the corresponding solution lives in $\mathcal{H}$ (by Lemma \ref{lem_adm} in Appendix \ref{sec_app_B}).

Theorem~\ref{thm_main} provides a natural exact controllability framework in $V$.
A tempting interpretation is that $V$ might represent the ``intrinsic state space'' of the exactly controllable dynamics.
This is misleading: although the free semigroup $T_0(t)$ maps $V$ into itself, HUM trajectories (i.e., trajectories minimizing the $L^2$ norm of the control) associated with exact controllability may leave $V$ for intermediate times.

\begin{proposition}\label{prop:noninv}
Assume \eqref{assumption_gamma_n} and let $T>2L$.
Then $T_0(t)$ restricts to a bounded operator on $V$ for every $t\ge0$.
However, for any fixed $t\in(0,T)$, there exist $\mathcal{X}_0,\mathcal{X}_1\in V$ such that the HUM trajectory steering $\mathcal{X}_0$ to $\mathcal{X}_1$ in time $T$ is not in $V$ at time $t$.
\end{proposition}

Recall once again that the real-valued case follows by restriction; in particular, for real endpoints $\mathcal{X}_0,\mathcal{X}_1$, the HUM control can be chosen real (see Remark \ref{rem:real_trajectories}).

Note that the proof of the invariance of $V$ under $T_0(t)$, as written in the previous proposition, follows from the spectral expansion \eqref{eq: semigroup expansion} and the definition \eqref{def_V}.
The second claim is subtler and relies on the fact that the control-to-state map at intermediate times involves the observation coefficients $\psi_{1,n}^3(L)$, which decay exponentially with $n$; as a result, the intermediate HUM map from $V'$ to $V$ is unbounded.
A complete proof is provided in Appendix~\ref{sec_app_D}.

\begin{remark}
Proposition~\ref{prop:noninv} should be understood as follows: $V$ is the correct endpoint space for exact controllability (it captures the reachable set at time $T$), but it is not an invariant state space for the controlled trajectories generated by HUM.
\end{remark}

\subsubsection{Examples and genericity issues}\label{sec_examples_gamma_n}
In view of commenting on the dependence of the coefficients $\gamma_n$ (and hence of the spaces $V$ and $V_0$) on the coupling profile $\beta$, in this section, we assume that 
\begin{equation}\label{beta_char}
\beta(x) = \beta_0 \, \mathds{1}_{[a,b]}(x) \qquad \forall x\in(0,L)
\end{equation}
for some $\beta_0 \in\mathbb{R}\setminus\{0\}$ and $0 \leq a < b \leq L$.
Then \eqref{def_gamma_n} gives
\begin{multline}\label{gamma:def:remark}
\gamma_n 
= \tfrac{\beta_0}{\lambda_{1,n}^2 + \tfrac{n^2 \pi^2}{L^2}} \Big( - \tfrac{n\pi}{L} \sinh( \lambda_{1,n} b ) \cos\big( \tfrac{n\pi b}{L} \big) + \tfrac{n\pi}{L} \sinh( \lambda_{1,n} a ) \cos\big( \tfrac{n\pi a}{L} \big) \\
+ \lambda_{1,n} \cosh( \lambda_{1,n} b ) \sin\big( \tfrac{n\pi b}{L} \big) 
- \lambda_{1,n} \cosh ( \lambda_{1,n} a ) \sin\big( \tfrac{n\pi a}{L} \big) \Big) 
\end{multline}
if $\lambda_{1,n}\neq0$, and 
$\gamma_n = - \tfrac{\beta_0 L}{n\pi} \left( b \cos\left(\tfrac{n\pi}{L} b \right) - a \cos\left(\tfrac{n\pi}{L} a \right) \right) + \tfrac{\beta_0 L^2}{n^2\pi^2} \left( \sin\left(\tfrac{n\pi}{L} b \right) - \sin\left(\tfrac{n\pi}{L} a \right) \right) $
if $\lambda_{1,n}=0$ (i.e., $\tfrac{n^2 \pi^2}{L^2} = c$).
In particular, for each $n$, $\gamma_n$ is an analytic function of $(a,b)$.
Moreover, one has
\begin{equation*}
\gamma_n 
= -\tfrac{\beta_0L^2}{2n^2\pi^2}\,e^{(\frac{n^2\pi^2}{L^2}-c)b}
\Big( \! \sin(\tfrac{n\pi}{L}b-\theta_n)-e^{-(\frac{n^2\pi^2}{L^2}-c)(b-a)}\sin(\tfrac{n\pi}{L}a-\theta_n) \Big) \Big(1+\mathrm{O}\Big(\frac{1}{n^2}\Big)\Big) 
\end{equation*}
where $\theta_n = -\arctan\big(\tfrac{n\pi}{L\lambda_{1,n}}\big) = \frac{L}{n\pi}+ \mathrm{O}\Big(\frac{1}{n^3}\Big)$ as $n\to+\infty$.

Recalling that the parabolic weights appearing in the norms of $V$ and $V_0$ are of the form
$w_n=\frac{n^4}{\gamma_n^2}e^{\nu n^2}$, we infer that, as $n\to+\infty$,
\begin{equation}\label{eq:w_n}
w_n = \frac{4\pi^4}{\beta_0^2L^4}\,n^8\, \frac{\exp\!\big((\nu-\frac{2\pi^2}{L^2}b)n^2+2cb\big)} {\big|\sin(\tfrac{n\pi}{L}b-\theta_n)-e^{-(\frac{n^2\pi^2}{L^2}-c)(b-a)}\sin(\tfrac{n\pi}{L}a-\theta_n)\big|^2} \,(1+\mathrm{o}(1)).
\end{equation}
Hence $w_n\geq\mathrm{Cst}\, n^8 e^{(\nu-\frac{2\pi^2}{L^2}b)n^2}$: the parabolic weights are at least exponentially large (unless $\nu=\frac{2\pi^2}{L}$ and $b=L$, see Remark \ref{rem_b=L} below), i.e., the space $V$ (likewise $V_0$) is very small in the parabolic component.
The formula \eqref{eq:w_n} also shows that the weights $w_n$ may strongly oscillate as $n$ varies, depending on the values of $a$, $b$ and $L$.

\begin{remark}[Possible arithmetic amplification of the weights]\label{rem:arithmetics}
The oscillatory denominator in \eqref{eq:w_n} quantifies the possible large fluctuations of the weights. 
In this denominator, the second sine term is of order $\exp(-\mathrm{Cst}\,n^2)$ as $n\to+\infty$.
Hence, for large $n$, any near-cancellation in the denominator can only occur if $\sin(\tfrac{n\pi}{L}b-\theta_n)$ is itself extremely small, namely of order $\exp(-\mathrm{Cst}\,n^2)$ along some subsequence.
This requires the ratio $b/L$ to admit exceptionally accurate rational approximations.
Such configurations are non-generic and may lead to dramatic (super-exponential) growth of the weights,
whereas for typical parameters the denominator stays away from $0$ up to polynomial factors and
$w_n$ fluctuates only within the baseline exponential scale $n^8 e^{(\nu-\frac{2\pi^2}{L^2}b)n^2}$.
We do not pursue a finer arithmetic classification here.
\end{remark}

\begin{remark}\label{rem_b=L}
Assume that $b=L$. 
Since, for the null-controllability space $V_0$, we have $\nu=\frac{2\pi^2}{L}$, 
the exponential numerator in \eqref{eq:w_n} cancels and since $|\sin(n\pi-\theta_n)|\sim \frac{L}{n\pi}$, we obtain $w_n \sim \mathrm{Cst}\,n^{10}$ as $n\to+\infty$, 
Therefore, using \eqref{def_V}, up to norm equivalence we have 
$$
V_0 = (H^5(0,L)\cap H^1_0(0,L))\times H^1_{(0)}(0,L)\times L^2(0,L)
$$
where the boundary conditions on the higher-order derivatives of the first component are encoded by the sine basis. Hence, in this case, $V_0$ is a ``classical" Sobolev space.
In contrast, the space $V$ in \eqref{def_V} contains the additional exponential amplification $e^{2\pi^2 n^2 T/L^2}$ and is therefore much smaller.
\end{remark}

\begin{remark}\label{rem:beta_constant}
In the special case $a=0$, $b=L$ in \eqref{beta_char} (i.e., $\beta\equiv\beta_0$), one has
$$
\gamma_n = 
\left\{ \begin{array}{ll}
\tfrac{(-1)^{n}\beta_0 \tfrac{n\pi}{L}}{\big( \tfrac{n^2 \pi^2}{L^2} - c \big)^2 + \tfrac{n^2 \pi^2}{L^2}} \sinh\Big(\big( \tfrac{n^2 \pi^2}{L^2} - c \big) L \Big) \neq 0  & \text{if }\tfrac{n^2 \pi^2}{L^2}\neq c , \\[5mm]
\frac{(-1)^{n+1}\beta_0 L^2}{n\pi} \neq 0 & \text{if }\frac{n^2 \pi^2}{L^2} = c .
\end{array}\right.
$$
Hence 
\eqref{assumption_gamma_n} holds.
\end{remark}

\medskip
For illustration, Figure \ref{fig:gamma2} shows the dependence of $\gamma_2$ on $b$ for $a=0$, $L=1$ and $c = 50$. 
We have $\gamma_2=0$ for $b\approx 0.586$, thus controllability is lost for this value of $b$. 
\begin{figure}[ht]
\centering
\includegraphics[width=3.5in]{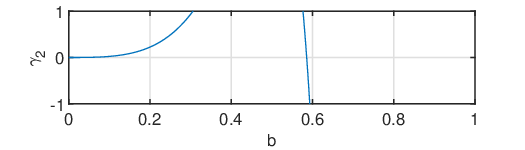}
\caption{$\gamma_2$ as a function of $b\in[0,1]$ for $L=1$, $c=50$ and $a=0$.}
\label{fig:gamma2}
\end{figure}

\medskip
We conclude this discussion with a genericity result for the family \eqref{beta_char}, showing also that the controllability properties are highly sensitive to perturbations of $\beta$.
Let
$$
S=\{(a,b)\,\mid\,0\leq a<b\leq L\},
\qquad
\hat S=\{(a,b)\in S\,\mid\,\gamma_n(a,b)\neq0\ \text{for all }n\geq1\}.
$$

\begin{lemma}\label{lem:genericity}
The set $\hat S$ is dense in $S$ and has full Lebesgue measure in $S$.
\end{lemma}

\begin{proof}
Fix $n\geq1$ and define $Z_n=\{(a,b)\in S\,|\,\gamma_n(a,b)=0\}$.
We have seen that the map $(a,b)\mapsto\gamma_n(a,b)$ is analytic.
It is not identically zero (e.g., for $(a,b)=(0,L)$ we have $\gamma_n\neq0$ by Remark~\ref{rem:beta_constant}), hence $Z_n$ has empty interior.
Moreover, $Z_n$ has zero Lebesgue measure (zero set of a nontrivial analytic function in dimension two).
Therefore $Z=\cup_{n\geq1}Z_n$ has zero measure.
Since $S\setminus Z_n$ is open and dense for each $n$, the Baire theorem implies that $S\setminus Z=\cap_{n\geq1}(S\setminus Z_n)$ is dense.
But $S\setminus Z=\hat S$, which concludes.
\end{proof}

\begin{remark}\label{rem_nonrobust}
Assume $T>2L$.
By Theorem~\ref{thm_main}, exact controllability in time $T$ in the space $V$ holds if and only if $(a,b)\in\hat S$.
In particular, exact controllability is generic within the parameter set $S$.
However, the set $\hat S$ does not need to be open, and therefore controllability may fail to be robust with respect to perturbations of $(a,b)$.

Moreover, even when $(a,b)\in\hat S$ (so that $\gamma_n(a,b)\neq0$ for all $n$), one should not expect any uniform lower bound on $|\gamma_n(a,b)|$ with respect to perturbations of $(a,b)$.
Indeed, for each fixed $n$, the map $(a,b)\mapsto \gamma_n(a,b)$ is analytic on $S$
and is not identically zero, so its zero set $Z_n$ (when nonempty) has empty interior.
Moreover, since $\gamma_n(a,b)\to 0$ as $b\downarrow a$, one can find parameters $(a,b)$
arbitrarily close to the diagonal $\{a=b\}$ for which $|\gamma_n(a,b)|$ is arbitrarily small.
Because $\hat S$ is dense in $S$ (Lemma~\ref{lem:genericity}), these small values can be achieved
while still keeping $(a,b)\in\hat S$.
This produces arbitrarily large weights in the spectral definitions of $V$ and $V_0$
(and hence potentially very large control costs), providing a quantitative interpretation of the lack of robustness.
\end{remark}

\section{Controllability of a heat--wave cascade}\label{sec:heat_wave}
We now turn to the reversed cascade \eqref{eq:HW}, for $t>0$, $x\in(0,L)$ and $u\in L^2(0,T)$.
The objective is to compare the controllability mechanisms with those of the wave--heat cascade \eqref{eq:WH} studied in Section~\ref{sec:wave_heat}, while avoiding unnecessary repetitions.

\subsection{Abstract formulation}
Setting $\mathcal{X}=(y,z,\partial_t z)^\top$ and working in the same pivot space $\mathcal{H}$ as in Section~\ref{sec:wave_heat}, the control system \eqref{eq:HW} can be written as
$\dot{\mathcal{X}}(t)=\mathcal{A}_{\mathrm{HW}} \mathcal{X}(t)+\mathcal{B}_{\mathrm{HW}} u(t)$, where $\mathcal{A}_{\mathrm{HW}}:D(\mathcal{A}_{\mathrm{HW}})\subset\mathcal{H}\rightarrow\mathcal{H}$ is defined by
$$
\mathcal{A}_{\mathrm{HW}} = \begin{pmatrix}
\partial_{xx}+c\,\mathrm{id} & 0 & 0\\
0 & 0 & \mathrm{id}\\
\beta\,\mathrm{id} & \partial_{xx} & 0
\end{pmatrix}
$$
with the same domain $D(\mathcal{A}_{\mathrm{HW}})=D(\mathcal{A})$ as the operator $\mathcal{A}$ studied in Section~\ref{sec:wave_heat}, and the (unbounded) control operator $\mathcal{B}_{\mathrm{HW}}$ associated with the Dirichlet boundary control on the heat equation can be defined by transposition, as in Appendix~\ref{sec_app_A}. In the adjoint system, the corresponding observation is the right heat flux:
\begin{equation}\label{eq:BHWstar}
\mathcal{B}_{\mathrm{HW}}^*(\psi_1,\psi_2,\psi_3)^\top=-\partial_x\psi_1(L),
\qquad \forall (\psi_1,\psi_2,\psi_3)^\top\in D(\mathcal{A}_{\mathrm{HW}}^*).
\end{equation}

\subsection{Spectral coefficients and controllability spaces}
The operator $\mathcal{A}_{\mathrm{HW}}$ is again Riesz-spectral.
Its eigenvalues are the same as $\mathcal{A}$, namely, $\lambda_{1,n}=c-\tfrac{n^2\pi^2}{L^2}$ for $n\in\mathbb{N}^*$ and $\lambda_{2,m}= i\,\tfrac{(2m+1)\pi}{2L}$ for $m\in\mathbb{Z}$.
Let $\{\phi_{1,n}\}_{n\geq1}\cup\{\phi_{2,m}\}_{m\in\mathbb{Z}}$ be the corresponding (biorthogonal) Riesz basis of eigenvectors of $\mathcal{A}_{\mathrm{HW}}$, and let $\{\psi_{1,n}\}_{n\geq1}\cup\{\psi_{2,m}\}_{m\in\mathbb{Z}}$ be the dual basis of eigenvectors of $\mathcal{A}_{\mathrm{HW}}^*$. They are different from the eigenvectors of $\mathcal{A}$ and $\mathcal{A}^*$. In particular, the dual Riesz basis is given by $\psi_{1,n} = ( \psi_{1,n}^1, 0 , 0 )$ with
\begin{equation*}
\psi_{1,n}^1(x) = \sqrt{\frac{2}{L}} \sin\left( \frac{n\pi}{L} x \right)
\end{equation*}
and $\psi_{2,m} = ( \psi_{2,m}^1 , \psi_{2,m}^2 , \psi_{2,m}^3 )$ with
\begin{align*}
\psi_{2,m}^1(x) &= \tfrac{1}{\sqrt{L}\,\overline{r_m}\sinh(\overline{r_m}L)}\Big(\sinh(\overline{r_m}(L-x))\int_0^x \beta(s)\sinh(\lambda_{2,m}s)\sinh(\overline{r_m}s)\,ds \\
&\qquad\qquad\qquad\qquad\quad +\sinh(\overline{r_m}x)\int_x^L \beta(s)\sinh(\lambda_{2,m}s)\sinh(\overline{r_m}(L-s))\,ds\Big), \\
\psi_{2,m}^2(x) &= \tfrac{1}{\lambda_{2,m}\sqrt{L}}\sinh(\lambda_{2,m}x),\\
\psi_{2,m}^3(x) &= \tfrac{1}{\sqrt{L}}\sinh(\lambda_{2,m}x),
\end{align*}
where $\overline{r_m}^2 = \overline{\lambda_{2,m}}-c$ with $\mathrm{Re}(r_m)>0$. It follows, in particular, that,
\begin{equation}\label{eq:HW_obs_heat}
\mathcal{B}_{\mathrm{HW}}^*\psi_{1,n}=-\partial_x\psi_{1,n}^1(L)=(-1)^{n+1}\sqrt{\tfrac{2}{L}}\tfrac{n\pi}{L}\neq 0
\qquad \forall n\geq1,
\end{equation}
where $\psi_{1,n}=(\psi_{1,n}^1,\psi_{1,n}^2,\psi_{1,n}^3)^\top$, showing that any parabolic mode is controllable (as expected, since the heat equation \eqref{eq:HW_y} with Dirichlet right-boundary control is exactly null controllable); for the hyperbolic modes the coupling through $\beta$ enters the observation via the coefficients
\begin{equation}\label{def_Gamma_m}
\Gamma_m=\int_0^L\beta(s)\,\sinh(\lambda_{2,m}s)\,\sinh(\overline{r_m}\,s)\,\mathrm{d}s,
\end{equation}
where $r_m$ is a square root of $\overline{\lambda_{2,m}}-c$ with nonnegative real part,
and one has 
\begin{equation}\label{eq:HW_obs_wave}
\mathcal{B}_{\mathrm{HW}}^*\psi_{2,m}=-\partial_x\psi_{2,m}^1(L)=\tfrac{\Gamma_m}{\sqrt{L}\,\sinh(\overline{r_m}L)}.
\end{equation}
In particular, a necessary condition for approximate controllability is
\begin{equation}\label{eq:Gamma_nonzero}
\Gamma_m\neq0\qquad \forall m\in\mathbb{Z}.
\end{equation}
This condition \eqref{eq:Gamma_nonzero} plays for the heat--wave cascade \eqref{eq:HW} the symmetric role as condition \eqref{assumption_gamma_n} does for the wave--heat cascade \eqref{eq:WH}.

Noting that $|\sinh(\overline{r_m}L)|^2 \sim \frac{1}{4}\exp\!\big(\sqrt{2|m|\pi L}\big)$ as $|m|\to\infty$,
arguing as in Section~\ref{sec:controllability_WHT} and applying Theorem~\ref{thm_IM}, the observation coefficients \eqref{eq:HW_obs_wave} suggest to define the observation space
\begin{multline*}
V_{\mathrm{HW}}' = \bigg\{ \sum_{n\geq1} a_n\psi_{1,n}+\sum_{m\in\mathbb{Z}} b_m\psi_{2,m}\quad \Big|\quad a_n\in\mathbb{C},\ b_m\in\mathbb{C} , \\
\sum_{n\geq1} n^2 e^{-\sigma n^2} \, |a_n|^2 + \sum_{m\in\mathbb{Z}}|\Gamma_m|^2 e^{-\sqrt{2|m|\pi L}}\,|b_m|^2<\infty \bigg\}, 
\end{multline*}
where $\sigma =\frac{2\pi^2}{L^2}T$, endowed with the obvious norm.
The dual $V_{\mathrm{HW}}$ of $V_{\mathrm{HW}}'$ with respect to the pivot space $\mathcal{H}$ is 
\begin{multline*}
V_{\mathrm{HW}} = \bigg\{ \sum_{n\geq1} a_n\phi_{1,n}+\sum_{m\in\mathbb{Z}} b_m\phi_{2,m}\quad \Big|\quad a_n\in\mathbb{C},\ b_m\in\mathbb{C},\\
\sum_{n\geq1}\frac{e^{\sigma n^2}}{n^2} \, |a_n|^2
+\sum_{m\in\mathbb{Z}}\frac{e^{\sqrt{2|m|\pi L}}}{|\Gamma_m|^2}\,|b_m|^2<\infty \bigg\}.
\end{multline*}
The Hilbert spaces $V_{0,\mathrm{HW}}$ and $V_{0,\mathrm{HW}}'$ are defined similarly as $V_{\mathrm{HW}}$ and $V_{\mathrm{HW}}'$ but with $\sigma=0$ (i.e., no exponential parabolic weight), so the parabolic component of $V_{0,\mathrm{HW}}$ is comparable to $H^{-1}(0,L)$, as expected because $H^{-1}(0,L)$ is the exact null-controllability space for the boundary-controlled heat equation (see \cite{trelat_SB,tucsnakweiss}).
\subsection{Main result on the heat-wave cascade}

\begin{theorem}\label{thm:HW}
Assume $T>2L$ and \eqref{eq:Gamma_nonzero}. Then:
\begin{enumerate}
\item \label{thm_HW_obs} 
There exists $C_T>0$ such that 
$\int_0^T\big|\mathcal{B}_{\mathrm{HW}}^*\mathcal{X}(t)\big|^2\,dt\geq C_T\|\mathcal{X}_0\|_{V_{\mathrm{HW}}'}^2$
for every solution of $\dot{\mathcal{X}}(t)=\mathcal{A}_{\mathrm{HW}}^*\mathcal{X}(t)$. 
%
\item \label{thm_HW_finitetimeobs} 
There exists $C_T^0>0$ such that 
$\int_0^T\big|\mathcal{B}_{\mathrm{HW}}^*\mathcal{X}(t)\big|^2\,dt\geq C_T^0\|\mathcal{X}(T)\|_{\mathcal{H}}^2$
for every solution of $\dot{\mathcal{X}}(t)=\mathcal{A}_{\mathrm{HW}}^*\mathcal{X}(t)$.
%
\item \label{thm_HW_exact} 
The system \eqref{eq:HW} is exactly controllable in time $T$ in the space $V_{\mathrm{HW}}$.
\item \label{thm_HW_null} 
The system \eqref{eq:HW} is exactly null controllable in time $T$ in the space $V_{0,\mathrm{HW}}$.
\item \label{thm_HW_approx} 
The system \eqref{eq:HW} is approximately controllable in time $T$ in $\mathcal{H}$
(and in any Hilbert space continuously and densely embedded in $\mathcal{H}$).
\end{enumerate}
Moreover:
\begin{itemize}
\item If $T<2L$, then \eqref{eq:HW} is not approximately controllable in time $T$ in $\mathcal{H}$
(and in any Hilbert space continuously and densely embedded in $\mathcal{H}$).
\item If \eqref{eq:Gamma_nonzero} fails (i.e., $\Gamma_{m_0}=0$ for some $m_0$), then \eqref{eq:HW} is not approximately controllable in any time $T>0$ in any Hilbert space continuously and densely embedded in $\mathcal{H}$.
\end{itemize}
\end{theorem}

\begin{proof}
The proof follows the same lines as that of Theorem~\ref{thm_main}, with the roles of the parabolic and hyperbolic families exchanged in the observation.
The condition \eqref{eq:Gamma_nonzero} is the analogue of \eqref{assumption_gamma_n} and guarantees that each hyperbolic mode is effectively ``seen'' by the heat boundary flux through the internal coupling.
The time threshold $T>2L$ is again imposed by the wave component.
\end{proof}

\begin{remark}\label{rem_Gamma_m}
Since $|\sinh(\lambda_{2,m}s)|\leq 1$ and $|\sinh(\overline{r_m}s)|\leq \cosh(\mathrm{Re}(r_m)s)$, and noting that $\mathrm{Re}(r_m)\sim\sqrt{\frac{|m|\pi}{2L}}$ as $|m|\to\infty$, we obtain
\begin{equation}\label{estim_Gamma_m}
|\Gamma_m|^2 e^{-\sqrt{2|m|\pi L}} \leq \frac{\mathrm{Cst}}{|m|+1} .
\end{equation}
Consequently, the hyperbolic weights in $V_{\mathrm{HW}}'$ are at most of polynomial order, whereas the weights $\frac{e^{\sqrt{2|m|\pi L}}}{|\Gamma_m|^2}$ in $V_{\mathrm{HW}}$ grow at least like $|m|$ (and may be much larger if $|\Gamma_m|$ is small for some modes).
This behavior contrasts with the wave--heat cascade (Remark~\ref{rem:spaceV}), where coupling-dependent weights act on the parabolic modes and may oscillate between distinct exponential scales.

Actually, if $\beta\equiv 0$ on $(b,L]$ for some $b<L$ (i.e., the coupling is supported away from the boundary $x=L$), then \eqref{estim_Gamma_m} can be improved, by noting that
$|\Gamma_m|
\leq  \|\beta\|_{L^\infty}\int_0^b \cosh(\mathrm{Re}(r_m)s)\,ds
\leq  \mathrm{Cst}\,\frac{e^{\mathrm{Re}(r_m)b}}{\mathrm{Re}(r_m)}$.
Hence, setting $\kappa=\sqrt{\frac{2\pi}{L}}(L-b)$,
\begin{equation*}
|\Gamma_m|^2 \, e^{-\sqrt{2|m|\pi L}} \leq \frac{\mathrm{Cst}}{|m|+1}e^{-\kappa\sqrt{|m|}}
\qquad \forall m\in\mathbb{Z}.
\end{equation*}

It is proved in Appendix \ref{app_Gamma_m} that, if $\beta\in W^{k+1,\infty}(0,L)$ for some $k\in\mathbb{N}$ and if $\beta^{(j)}(L)=0$ for $j=0,\dots,k-1$ and $\beta^{(k)}(L)\neq0$, then there exist $C_1,C_2>0$ such that, for every $m\in\mathbb{Z}$,
\begin{equation}\label{two_sided_Gamma_m}
\frac{C_1}{|m|^{p_k}+1} \leq |\Gamma_m|^2\, e^{-\sqrt{2|m|\pi L}}  \leq  \frac{C_2}{|m|^{p_k}+1} 
\quad
\textrm{with}\ \ 
p_k = \begin{cases}
2k+3 & \text{if }k\text{ is even},\\
2k+2 & \text{if }k\text{ is odd},
\end{cases}
\end{equation}
and therefore the hyperbolic weights in $V_{\mathrm{HW}}$ are equivalent, up to constants, to $(|m|^{p_k}+1)$ in the wave eigenbasis, implying that the wave component of $V_{\mathrm{HW}}$ is $H^{1+p_k/2}(0,L)\times H^{p_k/2}(0,L)$ (up to norm equivalence), with Dirichlet boundary condition at $0$ and Neumann boundary condition at $L$. 
%
\end{remark}

\subsection{Comparison with the wave--heat cascade}

The two cascades share the same hyperbolic minimal time $T>2L$ but differ in the way the coupling affects the controllability space.

\begin{itemize}
\item In the wave--heat cascade \eqref{eq:WH}, the observation coefficient on the parabolic modes involves the quantities $\gamma_n$ defined in \eqref{def_gamma_n}, which combine the internal coupling $\beta$ with the exponentially growing factor $\sinh(\lambda_{1,n}\cdot)$.
This may generate very irregular weights in the controllability space $V$ and leads to generic but potentially non-robust controllability.
\item In the heat--wave cascade \eqref{eq:HW}, the parabolic observation weights \eqref{eq:HW_obs_heat} are explicit and exponential in $n$, while the coupling enters through the hyperbolic weights (which depend on $\Gamma_m$).
By Remark \ref{rem_Gamma_m}, we always have the coarse estimate $\vert\Gamma_m\vert^{-2}e^{-\sqrt{2\vert m\vert\pi L}} \geq\mathrm{Cst}\,(|m|+1)$. Under mild regularity assumptions on $\beta$, we have a two-sided polynomial estimate: for instance if $\beta\in W^{1,\infty}(0,L)$ and $\beta(L)\neq 0$ then $\vert\Gamma_m\vert^{-2}e^{-\sqrt{2\vert m\vert\pi L}} \asymp \mathrm{Cst}\,(|m|^3+1)$ and the wave component of $V_{\mathrm{HW}}$ is comparable to $H^{\frac{5}{2}}(0,L)\times H^{\frac{3}{2}}(0,L)$.

In contrast, if $\beta$ is supported away from $x=L$, then $\Gamma_m$ is much smaller and the weights become very small, reflecting large control costs.
\end{itemize}

\medskip
Besides, the heat-wave cascade \eqref{eq:HW} is proved controllable for $T>2L$ by an Ingham-M\"untz argument, but the regime $T<2L$ remains open.
A particularly interesting case is $\beta\equiv1$, for which one may conjecture improved controllability properties in small time, possibly at the price of a controllability space depending on $T$.
A refined description in the spirit of partial observability for the wave equation (see \cite{dehman2025regional}) or a Gramian-based characterization of the reachable set could be relevant.
Transmutation ideas may also provide a useful viewpoint for this question (see, e.g., \cite{ErvedozaZuazua_ARMA2011}).

\section{Conclusion and open problems}\label{sec_conclusion}
We analyzed controllability properties for coupled wave/heat cascades on a one-dimensional interval.
Using a fully spectral approach based on Riesz bases and a combined Ingham--M\"untz inequality, we obtained complete characterizations of exact (and null) controllability in weighted spaces.

For the wave--heat cascade \eqref{eq:WH} (wave controlled at the boundary, heat driven through the internal coupling), exact controllability holds in any time $T>2L$ provided that the modal coefficients $\gamma_n$ defined in \eqref{def_gamma_n} do not vanish.
The corresponding controllability space $V$ is explicitly described in terms of $(\gamma_n)$ and may exhibit highly irregular weights; in particular, controllability is generic but can fail to be robust within natural coupling families.
We also established a structural limitation of the HUM framework: the space $V$ is not invariant along HUM trajectories, even though it is the correct endpoint space for exact controllability.

For the heat--wave cascade \eqref{eq:HW} (heat controlled at the boundary, wave driven through the coupling), we derived an analogous characterization in the space $V_{\mathrm{HW}}$ under the non-vanishing condition \eqref{eq:Gamma_nonzero}.
A key qualitative difference is that reversing the direction of the coupling transfers the coupling-dependent weights from the parabolic modes (through $\gamma_n$) to the hyperbolic modes (through $\Gamma_m$), which typically results in a more classical (Sobolev-type) behavior for the wave component.

\medskip

Several questions remain open and would deserve further investigation.

\medskip
\noindent{\bf Mixed spectral inequalities and sharp spaces.}

\medskip
\noindent\emph{-- The critical time $T=2L$} is not covered by the present Ingham-M\"untz approach.
A sharp description at $T=2L$ would require a refined mixed hyperbolic-parabolic inequality in a critical regime.

\medskip
\noindent\emph{-- Almost optimality of the parabolic weights.}
The parabolic weights in Appendix \ref{sec_IM} are ``almost optimal''.
It is open whether one can combine sharper Lebeau-Robbiano type weights with the hyperbolic gap while preserving a usable observability estimate; any such improvement would enlarge the exact controllability spaces.

\medskip
\noindent\emph{-- Block moment problems and eigenvalue multiplicities.}
Theorem \ref{thm_IM} is stated for sums of simple exponentials.
If some eigenvalues have finite algebraic multiplicity, the moment method naturally leads to exponential polynomials of the form $t^\ell e^{\lambda t}$, and one needs block or vectorial versions of Ingham-type inequalities.
Such extensions are closely related to block moment methods developed to handle spectral condensation phenomena (see for instance \cite{benabdallah2020block}).
Adapting these ideas to mixed hyperbolic-parabolic spectra is a natural perspective.

\medskip
\noindent{\bf Higher dimensions and geometry.}

\medskip
\noindent\emph{-- Beyond 1D and mixed spectral inequalities.}
Our approach instrumentally relies on the Ingham-M\"untz inequality.
In higher dimensions, the wave spectrum is no longer uniformly gapped and one typically replaces Ingham-type estimates by microlocal observability inequalities (geometric control condition) or multiplier arguments.

A challenging open question is to develop mixed estimates that combine parabolic Carleman or Lebeau-Robbiano inequalities with hyperbolic microlocal observability.
From a methodological viewpoint, this suggests combining Carleman estimates and spectral inequalities for parabolic components with semiclassical or microlocal propagation estimates for hyperbolic components, possibly through a frequency splitting argument.

Even in an intermediate setting (a 1D wave equation coupled to a multi-D heat equation), one may hope to derive an ``Ingham-Lebeau-Robbiano'' estimate by interfacing nonharmonic Fourier series techniques with multi-D spectral inequalities.
Whether such a mixed inequality can be made sharp, and whether it yields optimal controllability spaces, remains open.

\medskip
\noindent\emph{-- Coupling geometry and disjoint interaction regions.}
For internally controlled multi-D heat-wave systems, null controllability in $H^{-1}\times H_0^1\times L^2$ is proved in \cite{fernandezcaradeteresa2004} under geometric assumptions and, crucially, an overlap between the control and coupling regions (essentially reducing the wave component to an internally controlled wave equation on $\omega\cap O$).
The case of disjoint regions $\omega\cap O=\emptyset$, or boundary control of the heat component, appears substantially more delicate and remains largely open and should exhibit genuinely new propagation and spectral phenomena.

\medskip
\noindent{\bf Model extensions and design questions.}


\medskip
\noindent\emph{--Time-dependent coefficients.}
Extending Riesz basis techniques to nonautonomous cascades with time-varying couplings $\beta(t,x)$ or variable coefficients
raises delicate issues: time-dependence may destroy the rigid spectral structure underlying nonharmonic Fourier methods.
Identifying classes of coefficients for which a quantitative controllability theory persists is an open direction.

\medskip
\noindent\emph{-- Nonlinear extensions.}
It would be natural to investigate semilinear cascades, where one expects local controllability under suitable smallness and well-posedness assumptions, possibly via fixed point arguments around the linear controllability spaces.
How the nonstandard spectral weights interact with nonlinearities (in particular regarding robustness and cost estimates) is open.

\medskip
\noindent\emph{-- More general coupled systems.}
The present paper focuses on prototypical 1D cascades coupling a hyperbolic and a parabolic equation.
There is a vast literature on indirect controllability and stabilization of coupled systems, for instance for weakly coupled hyperbolic systems and indirect observability methods (see, e.g., \cite{alabau2003twolevel,alabau2011indirect}), and for cascade or more general couplings of parabolic equations (see \cite{ammarkhodja2011survey,boyer} and typical cascade results such as \cite{gonzalezburgos2010cascade}).
These works suggest that understanding sharp controllability spaces and robust cost estimates remains a subtle issue well beyond the wave/heat prototypes, including for thermoelastic models \cite{lebeauzuazua1998thermo} and other mixed couplings.

We mention that, for non-cascade couplings, a possible approach might be to view such systems as bounded perturbations of a cascade configuration and combine mode-by-mode rank conditions (Hautus-type tests) with robustness properties of exact or approximate controllability for admissible control operators (see \cite{DuprezOlive} for such results). 

\medskip
\noindent\emph{-- Choosing the coupling $\beta$ for a better controllability.}
An interesting design question is to understand how the choice of the coupling profile $\beta$ influences the controllability spaces and the control cost.
For instance, in the context of the present paper, one may ask whether there exist optimal choices of $\beta$ maximizing suitable spectral nondegeneracy quantities such as $\inf_{n\geq1}\vert\gamma_n\vert$ or $\inf_{m\in\mathbb{Z}}\vert\Gamma_m\vert$, or minimizing the HUM cost operator
(see \cite{PTZ_AIHPC2013, PTZ_ARMA2015, PTZ_JEMS2016} for related studies).

On the same line, the spectral nondegeneracy assumption $\gamma_n\neq0$ for all $n$ may potentially be relaxed by allowing time-dependent coupling profiles $\beta=\beta(t,\cdot)$, or time-varying coupling supports.
This suggests studying controllability under switching or moving couplings, with the hope that each frequency is sufficiently excited on some time subinterval even if a fixed profile yields cancellations.


\section*{Acknowledgment}
We are indebted to Karine Beauchard, Franck Boyer, Nicolas Burq, Shirshendu Chowdhury, Belhassen Dehman, Sylvain Ervedoza, Philippe Jaming, J\'er\^ome Le Rousseau, Qi L\"u, Lassi Paunonen, Xu Zhang, and Enrique Zuazua for useful discussions, remarks and suggestions.


\appendix

\section{Defining $\mathcal{B}$ by transposition}\label{sec_app_A}
\noindent
At the beginning of Section~\ref{sec:controllability_WHT}, we wrote \eqref{eq:WH} in the abstract form \eqref{eq:abstract_WHT}.
In this appendix, we recall how to define the control operator $\mathcal{B}$ by transposition.
The adjoint operator $\mathcal{A}^*$ is given in Lemma~\ref{lem: adjoint operator A0*}.
Following \cite{trelat_SB,tucsnakweiss}, identifying $\mathcal{H}$ with its dual, we search $\mathcal{B}\in L(\mathbb{C},D(\mathcal{A}^*)')$, or equivalently, $\mathcal{B}^*\in L(D(\mathcal{A}^*),\mathbb{C})$, where $D(\mathcal{A}^*)'$ is the dual of $D(\mathcal{A}^*)$ with respect to the pivot space $\mathcal{H}$.
Note that $D(\mathcal{A}^*) \subset \mathcal{H}\subset D(\mathcal{A}^*)'$ with continuous and dense embeddings.

Since the equality $\dot{\mathcal{X}}(t)=\mathcal{A} \mathcal{X}(t)+\mathcal{B} u(t)$ is written in $D(\mathcal{A}^*)'$, using the duality bracket $\langle\ ,\ \rangle_{D(\mathcal{A}^*)',D(\mathcal{A}^*)}$, we have, for any $\psi\in D(\mathcal{A}^*)$,
\begin{equation*}
\langle\dot{\mathcal{X}}(t),\psi\rangle_{D(\mathcal{A}^*)',D(\mathcal{A}^*)}
= \langle \mathcal{X}(t),\mathcal{A}^*\psi\rangle_{\mathcal{H}} + u(t)\,\mathcal{B}^*\psi .
\end{equation*}
On the other hand, taking $\mathcal{X}(t,\cdot)=(y(t,\cdot),z(t,\cdot),\partial_tz(t,\cdot))^\top$ a sufficiently regular solution of \eqref{eq:WH} and denoting $\psi=(\psi_1,\psi_2,\psi_3)^\top$, integrations by parts yield
$$
\langle\dot{\mathcal{X}}(t),\psi\rangle_{\mathcal{H}} = \langle \mathcal{X}(t),\mathcal{A}^*\psi\rangle_{\mathcal{H}} + u(t)\,\psi_3(L).
$$
A density argument shows that
$$
\mathcal{B}^*\psi = \psi_3(L)\qquad\forall(\psi_1,\psi_2,\psi_3)^\top\in D(\mathcal{A}^*),
$$
which is \eqref{eq:B0star}.

\section{Well-posedness and admissibility property}\label{sec_app_B}
\noindent
We say that \eqref{eq:abstract_WHT} is well-posed in $\mathcal{H}$ if, for all $T>0$ and $u\in L^2(0,T)$, any mild solution with $\mathcal{X}(0)\in\mathcal{H}$ satisfies $\mathcal{X}(t)\in\mathcal{H}$ for all $t\in[0,T]$, i.e., $\mathcal{B}$ is admissible (see \cite{trelat_SB,tucsnakweiss}).
By \cite[Proposition 5.13 and Remark 5.56]{trelat_SB} or \cite[Theorem 4.4.3]{tucsnakweiss}, $\mathcal{B}$ is admissible if and only, for some $T>0$ (and equivalently, for any $T>0$), there exists $K_T>0$ such that 
$\int_0^T \vert \mathcal{B}^*\mathcal{X}(t) \vert^2\, \mathrm{d}t \leq K_T \Vert \mathcal{X}(0)\Vert_{\mathcal{H}}^2$, i.e., 
\begin{equation}\label{adm}
\int_0^T \vert \mathcal{X}^3(t,L)\vert^2\, \mathrm{d}t 
\leq K_T\left( \Vert \mathcal{X}^1(0)\Vert_{L^2}^2 + \Vert\partial_x \mathcal{X}^2(0)\Vert_{L^2}^2 + \Vert \mathcal{X}^3(0)\Vert_{L^2}^2 \right)
\end{equation}
for any solution $t\mapsto \mathcal{X}(t,\cdot)=(\mathcal{X}^1(t,\cdot),\mathcal{X}^2(t,\cdot),\mathcal{X}^3(t,\cdot))^\top$ of the dual system $\dot{\mathcal{X}}(t) = \mathcal{A}^*\mathcal{X}(t)$, i.e., of
\begin{subequations}\label{dual_system}
    \begin{align}
        &\partial_t \mathcal{X}^1 = \partial_{xx} \mathcal{X}^1 + c \mathcal{X}^1 , && \mathcal{X}^1(t,0) = \mathcal{X}^1(t,L) = 0, \label{dual_system_1} \\
        &\partial_t \mathcal{X}^2 = P_\beta \mathcal{X}^1 - \mathcal{X}^3 , && \mathcal{X}^2(t,0) = \partial_x \mathcal{X}^2(t,L) = 0, \label{dual_system_2} \\
        &\partial_t \mathcal{X}^3 = -\partial_{xx} \mathcal{X}^2 , && \mathcal{X}^3(t,0) = 0 . \label{dual_system_3} 
    \end{align}
\end{subequations}
Using \eqref{dual_system_2} and \eqref{dual_system_3}, we have 
$$
\partial_{tt}\mathcal{X}^3=\partial_{xx}\mathcal{X}^3-\beta \mathcal{X}^1,
\qquad 
\mathcal{X}^3(t,0) = \partial_x \mathcal{X}^3(t,L)=0
$$
the latter, because $\partial_t\partial_x \mathcal{X}^2(t,L)=0$) and $\mathcal{X}^2(t,x) = - \int_0^x \int_\tau^L \partial_t \mathcal{X}^3(s,x) \,\mathrm{d}s\,\mathrm{d}\tau$.
Since $\Vert\partial_x \mathcal{X}^2(0)\Vert_{L^2} = \Vert\partial_t \mathcal{X}^3(0)\Vert_{H^{-1}_{(0)}}$, where $H^{-1}_{(0)}(0,L)$ is the dual of $H^1_{(0)}(0,L)$ with respect to the pivot space $L^2(0,L)$, the admissibility inequality \eqref{adm} is equivalent to
\begin{equation}\label{adm1}
\int_0^T \vert \mathcal{X}^3(t,L)\vert^2\, \mathrm{d}t 
\leq K_T\Big( \Vert \mathcal{X}^1(0)\Vert_{L^2}^2 + \Vert \mathcal{X}^3(0)\Vert_{L^2}^2 + \Vert\partial_t \mathcal{X}^3(0)\Vert_{H^{-1}_{(0)}}^2 \Big)
\end{equation}
for any solution of
\begin{subequations}\label{dual_system1}
\begin{align}
        &\partial_t \mathcal{X}^1 = \partial_{xx} \mathcal{X}^1 + c \mathcal{X}^1 ,  \label{dual_system1_1} \\
        &\partial_{tt} \mathcal{X}^3 = \partial_{xx} \mathcal{X}^3 - \beta \mathcal{X}^1 ,  \label{dual_system1_2}     \\
        &\mathcal{X}^1(t,0) = \mathcal{X}^1(t,L) = \mathcal{X}^3(t,0) = \partial_x \mathcal{X}^3(t,L) = 0 . \label{dual_system1_3} 
\end{align}
\end{subequations}

\begin{lemma}\label{lem_adm}
The admissibility inequality \eqref{adm1} holds.
Therefore, the control system \eqref{eq:abstract_WHT} is well-posed in $\mathcal{H}$.
\end{lemma}

\smallskip

\begin{proof}
We first establish \eqref{adm1} with $\mathcal{X}^1=0$. We have, then, a 1D wave equation with Dirichlet condition at the left boundary and Neumann condition at the right boundary. The eigenvalues of the corresponding Laplacian operator are $\lambda_k=\frac{\pi}{2L}+\frac{k\pi}{L}$, for $k\in\mathbb{N}$, with associated eigenfunctions $\sqrt{\tfrac{2}{L}}\sin(\lambda_k x)$. Now, expanding $\mathcal{X}^3(t,x)$ as a spectral Fourier series in this basis of eigenfunctions, and using the Parseval theorem as in \cite[Section 5.2.4.1]{trelat_SB}, \eqref{adm1} follows, with $\mathcal{X}^1=0$.

Now, we treat $\mathcal{X}^1$ as a source term in the wave equation \eqref{dual_system1_2}, and \eqref{adm} follows by the Duhamel formula.
\end{proof}

\begin{remark}\label{rem:max_adm_B0}
Lemma~\ref{lem_adm} is essentially sharp with respect to the hyperbolic component: any Sobolev-type enlargement that weakens the norm on the wave component destroys $L^2$-admissibility, already for the uncoupled wave equation. 

Indeed, take $\beta\equiv0$ (hence $X_1\equiv0$ in \eqref{dual_system1_2}. Setting $\varphi_k(x)=\sqrt{\frac{2}{L}}\sin\big( (k+\tfrac12)\tfrac{\pi}{L}x\big)$ for $k\in\mathbb{N}$, one has $\varphi_k(0)=0$, $\partial_x\varphi_k(L)=0$ and $|\varphi_k(L)|=1$.
The solution $\mathcal{X}_3$ of \eqref{dual_system1_2}-\eqref{dual_system1_3} with initial data $\mathcal{X}_3(0)=\varphi_k$ and $\partial_t\mathcal{X}_3(0)=0$ is $\mathcal{X}_3(t,L)=\varphi_k(L)\cos( (k+\tfrac12)\tfrac{\pi}{L} t)$ and satisfies $\int_0^T |\mathcal{X}_3(t,L)|^2\,dt\geq \frac{T}{4}$ for $k$ sufficiently large.
On the other hand, for any $s>0$, $\|\varphi_k\|_{H^{-s}(0,L)}\sim k^{-s}\|\varphi_k\|_{L^2}\to 0$.
Therefore, no estimate of the form \eqref{adm1} can hold if one replaces the $L^2$-term $\|\mathcal{X}_3(0)\|_{L^2}$ by a strictly weaker Sobolev norm (e.g., $H^{-s}$), and the operator $\mathcal{B}$ fails to be $L^2$-admissible in such enlarged state spaces.
\end{remark}

\section{Ingham--M\"untz inequality}\label{sec_IM}
The key technical ingredient underlying our observability and controllability results is a \emph{combined Ingham--M\"untz inequality} for non-harmonic Fourier series mixing a parabolic family and a hyperbolic family.
The systematic use of such mixed inequalities in the analysis of heat--wave couplings was pioneered in \cite{zhang2004polynomial} (see also \cite{zhang2003polynomial}), and our strategy follows the same spectral viewpoint.
In this section we recall the Ingham--M\"untz inequality used throughout the paper and we discuss the (almost) optimality of the weights that it produces.

\begin{theorem}\label{thm_IM}
Let $(\lambda_{1,n})_{n\geq1}$ be a sequence of distinct complex numbers such that there exists $\alpha>1$ with $(\lambda_{1,n}/n^\alpha)_{n\geq1}$ bounded.
Assume that there exist $n_0\geq1$ and $C_1,C_2>0$ such that for all $n,n'\geq n_0$,
$$
-\mathrm{Re}(\lambda_{1,n}) \geq C_1\, |\mathrm{Im}(\lambda_{1,n})|
\qquad\text{and}\qquad
|\lambda_{1,n}-\lambda_{1,n'}|\geq C_2|n^\alpha-{n'}^\alpha|.
$$
Let $(\lambda_{2,m})_{m\in\mathbb{Z}}$ be another sequence of distinct complex numbers.
Assume that there exist $m_0\ge0$, $\gamma>0$, $z_0\in\mathbb{C}$ and $(\mu_m)_{m\in\mathbb{Z}}\in\ell^2(\mathbb{C})$ such that
$$
\lambda_{2,m} = \gamma\, i\,m + z_0 + \mu_m
\qquad \forall |m|\geq m_0,
$$
and that $\lambda_{1,n}\neq\lambda_{2,m}$ for all $n\geq1$, $m\in\mathbb{Z}$.
Then, for any $T>\frac{2\pi}{\gamma}$, there exists $C_T>0$ such that
\begin{multline}\label{IM}
\int_0^T \! \Big| \sum_{n\geq1} a_n e^{\lambda_{1,n}t} +\! \sum_{m\in\mathbb{Z}} b_m e^{\lambda_{2,m}t} \Big|^2  \mathrm{d}t 
\geq  C_T \bigg( \sum_{n\geq1} |a_n|^2 e^{2\mathrm{Re}(\lambda_{1,n})T} +\! \sum_{m\in\mathbb{Z}} |b_m|^2 \bigg)
\end{multline}
for all $(a_n)_{n\geq1},(b_m)_{m\in\mathbb{Z}}\in\ell^2(\mathbb{C})$.
\end{theorem}

\smallskip

Inequality \eqref{IM} can be interpreted as a superposition of the classical Ingham inequality \cite{Ingham} (purely hyperbolic spectra with a uniform gap) and of M\"untz--Sz\'asz type estimates for parabolic spectra \cite{avdoninivanov}.
In a control-theoretic form close to \eqref{IM}, it was first established in \cite{zhang2003polynomial,zhang2004polynomial} and later generalized in \cite{komornik2015ingham} (see also \cite{bhandari2024boundary}).

For the wave--heat cascade \eqref{eq:WH}, the spectrum is given in Lemma \ref{lem: eigenstructures of A0}, and the assumptions of Theorem~\ref{thm_IM} hold with $\alpha=2$ and $\gamma=\pi/L$ (and $\mu_m\equiv0$), and the time condition $T>2\pi/\gamma$ reduces to $T>2L$.

\begin{remark}[Almost optimality of the weights]\label{rem:IM_optimality}
The weights in \eqref{IM} are decisive, because they determine the size of the exact controllability spaces obtained by the method of moments.

\medskip
\noindent
\emph{(i) Hyperbolic part.}
Taking $a_n\equiv0$ in \eqref{IM} gives the classical Ingham inequality for $(\lambda_{2,m})_{m\in\mathbb{Z}}$.
In this purely hyperbolic setting, the family $\{e^{\lambda_{2,m}t}\}_{m\in\mathbb{Z}}$ forms a Riesz sequence in $L^2(0,T)$ whenever $T>2\pi/\gamma$, and one even has a \emph{reverse} inequality with the same $\ell^2$ weight on the coefficients, up to changing the constant (see \cite{avdoninivanov,Ingham}).
In particular, the unit weight $\sum_{m\in\mathbb{Z}}|b_m|^2$ on the right-hand side of \eqref{IM} is sharp.

\medskip
\noindent
\emph{(ii) Parabolic part.}
Taking $b_m\equiv0$ reduces \eqref{IM} to a M\"untz--Sz\'asz type inequality.
Under the polynomial separation assumptions of Theorem~\ref{thm_IM}, one has $\mathrm{Re}(\lambda_{1,n})\sim -c\, n^\alpha$ for some $c>0$, so the weight $e^{2\mathrm{Re}(\lambda_{1,n})T}$ is essentially of the form $e^{-c n^\alpha T}$ (up to changing constants).
For purely parabolic families, this weight is not sharp: one can typically replace it by a \emph{sub-Gaussian} weight $e^{-\mathrm{Cst}\, n^{\alpha/2}\sqrt{1+T}}$ (equivalently, $e^{-\mathrm{Cst}\,\sqrt{-\mathrm{Re}(\lambda_{1,n})}\sqrt{1+T}}$), see, e.g., \cite[Section~7]{russell1978controllability} and \cite[Section~6]{fernandezcarazuazua}.
Such refinements are closely related to Lebeau--Robbiano spectral inequalities (and to the sharp control cost for the heat equation).

\medskip
\noindent
\emph{(iii) Why \eqref{IM} is ``almost optimal''.}
In the purely parabolic setting, the exponent $1/2$ in $\sqrt{-\mathrm{Re}(\lambda_{1,n})}$ is essentially optimal in general geometries and cannot be replaced by any smaller power $\theta<1/2$ (see \cite{lebeau1995controle}).
Consequently, inequalities of the form \eqref{IM} cannot hold if one replaces $e^{2\mathrm{Re}(\lambda_{1,n})T}$ by a weaker weight $e^{-\mathrm{Cst}\, n^{c\alpha}}$ with $c<1/2$.
At the same time, it is (to our knowledge) an open question whether the \emph{combined} inequality \eqref{IM} remains valid when improving the parabolic weight from $e^{-c n^\alpha T}$ to $e^{-\mathrm{Cst}\, n^{\alpha/2}\sqrt{1+T}}$ while keeping the hyperbolic part.
Any such refinement would slightly enlarge the exact controllability spaces derived in this paper (for $\alpha=2$, it would replace Gaussian-type weights $e^{\nu n^2}$ by analytic-type weights $e^{\mathrm{Cst}\, n}$ in the parabolic component).
\end{remark}

\section{HUM non-invariance: proof of Proposition~\ref{prop:noninv}}\label{sec_app_D}

This appendix provides a detailed proof of Proposition~\ref{prop:noninv}.
We assume throughout that $T>2L$ and that the non-resonance condition \eqref{assumption_gamma_n} holds.

\subsection{Preliminaries}

For $t\in[0,T]$, define the control-to-state map
\begin{equation*}
L_t u=\int_0^t T_0(t-s)\mathcal{B} u(s)\,\mathrm{d}s,
\qquad u\in L^2(0,t).
\end{equation*}
By admissibility of $\mathcal{B}$ (Appendix~\ref{sec_app_B}), $L_t$ is bounded from $L^2(0,t)$ into $\mathcal{H}$.
Moreover, thanks to the observability inequality \eqref{exact_obs}, $L_T^*$ is bounded from $V'$ into $L^2(0,T)$, and therefore $L_T$ is bounded from $L^2(0,T)$ into $V$ (duality).
The HUM (or controllability Gramian) operator is
$G_T=L_T L_T^*:V'\to V$.
Exact observability implies that $G_T$ is an isomorphism between $V'$ and $V$ (see, e.g., \cite{lions,trelat_SB,tucsnakweiss}).

For $t\in(0,T)$, we also define the intermediate HUM map
\begin{equation*}
F_t=L_t L_T^*:V'\to \mathcal{H}.
\end{equation*}
Note that $F_t$ is bounded from $V'$ to $\mathcal{H}$ as a composition of bounded operators.

\subsection{Invariance of $V$ for the free semigroup}

\begin{lemma}\label{lem:V_invariant}
For every $t\ge0$, the semigroup $T_0(t)$ restricts to a bounded operator on $V$.
\end{lemma}

\begin{proof}
Let $\mathcal{X}\in V$ and expand $\mathcal{X}=\sum_{n\geq1}a_n\phi_{1,n}+\sum_{m\in\mathbb{Z}}b_m\phi_{2,m}$.
Then
$$
T_0(t)\mathcal{X}=\sum_{n\geq1}e^{\lambda_{1,n}t}a_n\phi_{1,n}+\sum_{m\in\mathbb{Z}}e^{\lambda_{2,m}t}b_m\phi_{2,m}.
$$
Since $\mathrm{Re}(\lambda_{1,n})\leq c$ and $|e^{\lambda_{2,m}t}|=1$, we infer from \eqref{def_norm_V} that $\|T_0(t)\mathcal{X}\|_V\leq e^{ct}\|\mathcal{X}\|_V$.
\end{proof}

\subsection{Unboundedness of $F_t$ from $V'$ to $V$}

For $n\geq1$, set
\begin{equation*}
\beta_n=\mathcal{B}^*\psi_{1,n}=\psi_{1,n}^3(L).
\end{equation*}
Lemma~\ref{lem_necessary_psi} provides the asymptotic behavior of $\beta_n$ as $n\to\infty$ in terms of $\gamma_n$.

\begin{lemma}\label{lem:F_unbounded}
For any fixed $t\in(0,T)$, the operator $F_t$ does \emph{not} map $V'$ continuously into $V$.
More precisely, $F_t$ is unbounded as a linear map from $V'$ to $V$.
\end{lemma}

\begin{proof}
Fix $t\in(0,T)$.
For each $n\geq1$, consider the control
\begin{equation}\label{eq:u_n_def}
u_n(s)=(L_T^*\psi_{1,n})(s)=\mathcal{B}^*T_0^*(T-s)\psi_{1,n}=\beta_n e^{\lambda_{1,n}(T-s)},\qquad s\in(0,T),
\end{equation}
and let $\mathcal{X}^{(n)}$ be the mild solution of \eqref{eq:abstract_WHT} with $\mathcal{X}^{(n)}(0)=0$ and input $u_n$.
By definition of $L_t$ and $L_T^*$, we have $\mathcal{X}^{(n)}(t)=F_t\psi_{1,n}$.

Let $x_{1,n}^{(n)}(t)=\langle \mathcal{X}^{(n)}(t),\psi_{1,n}\rangle$.
Taking the $\mathcal{H}$-inner product of \eqref{eq:abstract_WHT} with $\psi_{1,n}$ and using $\mathcal{A}^*\psi_{1,n}=\lambda_{1,n}\psi_{1,n}$ yields the scalar ODE
$$
\frac{\mathrm{d}}{\mathrm{d}t}x_{1,n}^{(n)}(t)=\lambda_{1,n}x_{1,n}^{(n)}(t)+\beta_n u_n(t),\qquad x_{1,n}^{(n)}(0)=0.
$$
Solving and using \eqref{eq:u_n_def} gives, for $\lambda_{1,n}\neq0$,
\begin{equation}\label{eq:x_n_formula}
x_{1,n}^{(n)}(t)=\beta_n^2\int_0^t e^{\lambda_{1,n}(t-s)}e^{\lambda_{1,n}(T-s)}\,\mathrm{d}s
=\frac{\beta_n^2}{2\lambda_{1,n}}\Big(e^{\lambda_{1,n}(T+t)}-e^{\lambda_{1,n}(T-t)}\Big).
\end{equation}
(For the finitely many indices such that $\lambda_{1,n}=0$, a similar computation yields an affine expression in $t$ and does not affect the argument.)

By \eqref{def_norm_V} and \eqref{def_norm_V'}, we have $\|\phi_{1,n}\|_V^2 = \frac{n^4}{\gamma_n^2}e^{\nu n^2}$ and $\|\psi_{1,n}\|_{V'}^2 = \frac{\gamma_n^2}{n^4}e^{-\nu n^2}$.
Since $x_{1,n}^{(n)}(t)$ is the coefficient of $\phi_{1,n}$ in the expansion of $\mathcal{X}^{(n)}(t)$, we have the lower bound
$
\|\mathcal{X}^{(n)}(t)\|_V^2\geq \frac{n^4}{\gamma_n^2}e^{\nu n^2} |x_{1,n}^{(n)}(t)|^2.
$
Therefore,
\begin{equation}\label{eq:ratio_def}
\frac{\|F_t\psi_{1,n}\|_V^2}{\|\psi_{1,n}\|_{V'}^2}
=\frac{\|\mathcal{X}^{(n)}(t)\|_V^2}{\|\psi_{1,n}\|_{V'}^2}
\geq \frac{n^8}{\gamma_n^4}e^{2\nu n^2}\,|x_{1,n}^{(n)}(t)|^2.
\end{equation}

We now estimate the right-hand side of \eqref{eq:ratio_def}.
For $n$ large enough, $\lambda_{1,n}<0$ and
$
|e^{\lambda_{1,n}(T+t)}-e^{\lambda_{1,n}(T-t)}|
=e^{\lambda_{1,n}(T-t)}\big(1-e^{2\lambda_{1,n}t}\big).
$
Since $\lambda_{1,n}\to-\infty$, we can fix $n_0$ such that $1-e^{2\lambda_{1,n}t}\geq \frac{1}{2}$ for all $n\geq n_0$.
Hence, from \eqref{eq:x_n_formula},
\begin{equation}\label{eq:x_lower}
|x_{1,n}^{(n)}(t)|^2\geq \frac{|\beta_n|^4}{16\,\lambda_{1,n}^2}\,e^{2\lambda_{1,n}(T-t)}\qquad \forall n\geq n_0.
\end{equation}

Moreover, the asymptotics \eqref{eq:psi_1n_asympt} implies that there exist constants $c_0>0$ and $n_1\geq n_0$ such that
\begin{equation}\label{eq:beta_lower}
|\beta_n|^2\geq c_0\,\frac{\gamma_n^2}{n^4}\,e^{-2\pi^2 n^2/L} \qquad \forall n\geq n_1.
\end{equation}
Combining \eqref{eq:ratio_def}, \eqref{eq:x_lower}, and \eqref{eq:beta_lower} yields 
$$
\frac{\|F_t\psi_{1,n}\|_V^2}{\|\psi_{1,n}\|_{V'}^2}
\geq\frac{c_0^2}{16}\,\frac{1}{\lambda_{1,n}^2}\exp\!\Big(\big(2\nu-\frac{4\pi^2}{L}\big)n^2+2\lambda_{1,n}(T-t)\Big) \qquad \forall n\geq n_1.
$$
Using $\nu=\frac{2\pi^2}{L}(1+\frac{T}{L})$ and $\lambda_{1,n}=c-\frac{n^2\pi^2}{L^2}$, we obtain
$$
\big(2\nu-\frac{4\pi^2}{L}\big)n^2+2\lambda_{1,n}(T-t)
= \frac{2\pi^2 (T+t)}{L^2}n^2+2c(T-t).
$$
Therefore,
$$
\frac{\|F_t\psi_{1,n}\|_V^2}{\|\psi_{1,n}\|_{V'}^2} \geq C\,\frac{1}{\lambda_{1,n}^2}\,e^{\frac{2\pi^2 (T+t)}{L^2}n^2}, \qquad \forall n\geq n_1,
$$
for some $C>0$ independent of $n$.
Since $\lambda_{1,n}\sim -\frac{\pi^2}{L^2}n^2$, the right-hand side diverges to $+\infty$ as $n\to\infty$.
This shows that $F_t$ is unbounded from $V'$ to $V$.
\end{proof}

\subsection{End of the proof of Proposition~\ref{prop:noninv}}
Fix $t\in(0,T)$.
Assume, by contradiction, that $F_t(V')\subset V$.
Since $F_t:V'\to \mathcal{H}$ is bounded, its graph as an operator $V'\to V$ is closed: if $\xi_k\to\xi$ in $V'$ and $F_t\xi_k\to \eta$ in $V$, then $F_t\xi_k\to F_t\xi$ in $\mathcal{H}$, hence $\eta=F_t\xi$ in $\mathcal{H}$ and therefore $\eta=F_t\xi$ in $V$ (injective embedding $V\hookrightarrow\mathcal{H}$).
By the closed graph theorem, $F_t$ would then be bounded from $V'$ to $V$, contradicting Lemma~\ref{lem:F_unbounded}.
Thus, there exists $\xi\in V'$ such that $F_t\xi\notin V$.

Let $\mathcal{X}_1=G_T\xi\in V$ and consider the HUM control $u=L_T^*\xi$, which drives $\mathcal{X}(0)=0$ to $\mathcal{X}(T)=\mathcal{X}_1$.
The corresponding trajectory satisfies
$$
\mathcal{X}(t)=L_t u = L_t L_T^*\xi = F_t\xi\notin V,
$$
which proves the non-invariance statement.

\section{Auxiliary estimates on the coefficients $\Gamma_m$}\label{app_Gamma_m}
In this appendix, we prove \eqref{two_sided_Gamma_m}.
We use the notations of Section~\ref{sec:heat_wave}.
Recall that $\lambda_{2,m}= i\,\tfrac{(2m+1)\pi}{2L}$, $r_m=\sqrt{\lambda_{2,m}-c}$ with $\mathrm{Re}(r_m)\geq  0$, and $\Gamma_m$ defined by \eqref{def_Gamma_m}, for any $m\in\mathbb{Z}$.
As $|m|\to\infty$, we have $r_m \sim e^{i\,\mathrm{sgn}(m)\pi/4}\sqrt{\tfrac{|m|\pi}{2L}}$ and $\mathrm{Re}(r_m)\sim \sqrt{\tfrac{|m|\pi}{2L}}$.
The notation $f\asymp g$ means that there exist $C_1,C_2>0$ such that $C_1\,g\leq f\leq C_2\,g$ for any $m\in\mathbb{Z}$.

\begin{proof}[Proof of \eqref{two_sided_Gamma_m}]
First, we note that
$\Gamma_m = \frac{i}{2}\int_0^L \beta(s)\,\sin(\tfrac{(2m+1)\pi}{2L} s)\,e^{r_m s}\,ds + \mathrm{O}(1)$ as $|m|\to\infty$. Setting $t=L-s$ gives
$$
e^{-r_m L}\Gamma_m
= \frac{i}{2}(-1)^m\int_0^L \beta(L-t)\,\cos\big(\tfrac{(2m+1)\pi}{2L} t\big)\,e^{-r_m t}\,dt + \mathrm{O}(e^{-\mathrm{Re}(r_m)L}).
$$
Since $\mathrm{Re}(r_m)\to+\infty$, the integral is localized at $t=0$.
Using the Taylor expansion 
$\beta(L-t)=\frac{(-1)^k}{k!}\beta^{(k)}(L)\,t^k + \mathrm{O}(t^{k+1})$ and $\int_0^{+\infty} t^k e^{-\alpha t}\,dt = \frac{k!}{\alpha^{k+1}}$, we obtain
$$
e^{-r_m L}\Gamma_m
= \frac{i}{2}(-1)^{m+k}\beta^{(k)}(L)\,\mathrm{Re}\big((-i\tfrac{(2m+1)\pi}{2L})^{-k-1}\big)
+ \mathrm{O}\big(|r_m-i\tfrac{(2m+1)\pi}{2L}|^{-k-2}\big).
$$
Now, we have $|r_m-i\tfrac{(2m+1)\pi}{2L}|\asymp |m|$ and $|r_m|\asymp |m|$ as $|m|\to\infty$.
Moreover, since
$$
(r_m-i\tfrac{(2m+1)\pi}{2L})^{-k-1}
=(-i\tfrac{(2m+1)\pi}{2L})^{-k-1}\big(1+\mathrm{O}(|m|^{-1/2})\big) ,
$$
the leading order of its real part depends on the parity of $k+1$:
if $k$ is odd, then $\mathrm{Re}\big((-i\frac{(2m+1)\pi}{2L})^{-k-1}\big)\asymp |m|^{-k-1}$, while
if $k$ is even, then $\mathrm{Re}\big((-i\frac{(2m+1)\pi}{2L})^{-k-1}\big)=0$, so the first nonzero term
comes from the $\mathrm{O}(|m|^{-1/2})$ correction, yielding an extra factor $|r_m|/|m|$ and we obtain
$\mathrm{Re}\big((r_m-i\frac{(2m+1)\pi}{2L})^{-k-1}\big) \asymp |m|^{-k-3/2}$. The result follows.
\end{proof}


\bibliographystyle{IEEEtranS}
\nocite{*}
\bibliography{IEEEabrv,mybibfile}

\end{document}